\documentclass[onefignum,onetabnum]{siamonline171218}

\pdfoutput=1

\usepackage{lipsum}
\usepackage{amsfonts}
\usepackage{graphicx}
\usepackage{epstopdf}
\usepackage{bm}
\usepackage{algorithm}
\usepackage{algpseudocode}
\usepackage{amsmath}
\usepackage{array}
\usepackage{booktabs}
\usepackage{subfig}
\usepackage{amsmath, amssymb, amsfonts}
\usepackage{diagbox}
\usepackage{appendix}
\usepackage{xr-hyper}
\usepackage{amsopn}

\newcommand{\norm}[1]{\Vert #1\Vert}
\newcommand{\mbr}{\mathbb{R}}
\newcommand{\mbn}{\mathbb{N}}
\newcommand{\gis}{G_\mathrm{IS}}
\newcommand{\gtis}{G_\mathrm{tIS}}
\newcommand{\ggis}{\tilde{G}_\mathrm{IS}}
\newcommand{\mbe}[1]{\mathbb{E}\left[#1\right]}
\newcommand{\simiid}{\stackrel{\text{iid}}{\sim}}
\newcommand{\R}{\mathbb{R}}

\ifpdf
  \DeclareGraphicsExtensions{.eps,.pdf,.png,.jpg}
\else
  \DeclareGraphicsExtensions{.eps}
\fi

\usepackage{enumitem}
\setlist[enumerate]{leftmargin=.5in}
\setlist[itemize]{leftmargin=.5in}

\newsiamremark{remark}{Remark}
\newsiamremark{hypothesis}{Hypothesis}
\crefname{hypothesis}{Hypothesis}{Hypotheses}
\newsiamthm{claim}{Claim}
\newsiamthm{Assumption}{Assumption}
\newsiamremark{Example}{Example}
\newsiamthm{Lemma}{Lemma}
\headers{}{Z. HE, H. WANG, X. WANG}

\title{Quasi-Monte Carlo and importance sampling methods for Bayesian inverse problems\thanks{Submitted to the editors DATE.
\funding{This work was funded by the National Natural Science Foundation of China (No. 12071154 and 
72071119), and the Guangdong Basic and Applied Basic Research Foundation (No. 2021A1515010275).}}}

\author{Zhijian He\thanks{School  of  Mathematics,   South  China  University  of  Technology,Guangzhou 510641, People's Republic of China
  (\email{hezhijian@scut.edu.cn}).} \and Hejin Wang\thanks{Corresponding author. Department of Mathematical Sciences, Tsinghua University, Beijing 100084,People's Republic of China 
  (\email{wanghj20@mails.tsinghua.edu.cn}).}
\and Xiaoqun Wang \thanks{Department of Mathematical Sciences, Tsinghua University, Beijing 100084,People's Republic of China 
  (\email{wangxiaoqun@mail.tsinghua.edu.cn}).}
}

\makeatletter
\newcommand*{\addFileDependency}[1]{
  \typeout{(#1)}
  \@addtofilelist{#1}
  \IfFileExists{#1}{}{\typeout{No file #1.}}
}
\makeatother

\newcommand*{\myexternaldocument}[1]{%
    \externaldocument{#1}%
    \addFileDependency{#1.tex}%
   \addFileDependency{#1.aux}%
}

\ifpdf
\hypersetup{
  pdftitle={Quasi-Monte Carlo and importance sampling methods for Bayesian inverse problems},
  pdfauthor={Z.HE, H.WANG, X. WANG}
}
\fi

\myexternaldocument{ex_supplement}

\begin{document}

\maketitle

\begin{abstract}
Importance Sampling (IS), an effective variance reduction strategy in Monte Carlo (MC) simulation, is frequently utilized for Bayesian inference and other statistical challenges. Quasi-Monte Carlo (QMC) replaces the random samples in MC with low discrepancy points and has the potential to substantially enhance error rates. In this paper, we integrate IS with a randomly shifted rank-1 lattice rule, a widely used QMC method, to approximate posterior expectations arising from Bayesian Inverse Problems (BIPs) where the posterior density tends to concentrate as the intensity of noise diminishes. Within the framework of weighted Hilbert spaces, we first establish the convergence rate of the lattice rule for a large class of unbounded integrands. This method extends to the analysis of QMC combined with IS in BIPs. Furthermore, we explore the robustness of the IS-based randomly shifted rank-1 lattice rule by determining the quadrature error rate with respect to the noise level. The effects of using Gaussian distributions and $t$-distributions as the proposal distributions on the error rate of QMC are comprehensively investigated. We find that the error rate may deteriorate at low intensity of noise when using improper proposals, such as the prior distribution. To reclaim the effectiveness of QMC, we propose a new IS method such that the lattice rule with $N$ quadrature points achieves an optimal error rate close to $O(N^{-1})$, which is insensitive to the noise level. Numerical experiments are conducted to support the theoretical results.

\end{abstract}



\section{Introduction}\label{sec1}

In recent years, Bayesian inference has become a key method for solving inverse problems and quantifying uncertainties \cite{ghanem2017handbook}. It utilizes the Bayesian formula to update the prior distribution of unknown parameters based on noisy observed data, thereby enhancing model training. This method is especially effective in diverse fields such as engineering and biological systems, where it plays a crucial role in handling uncertainties associated with complex problems.

One specific form of these complex problems can be considered as inferring the vector \(z \in \mathbb{R}^s\) from the vector \(y \in \mathbb{R}^J\). Here, \(z\) and \(y\) are related through a given response operator \(\mathcal{G}\). In this context, \(\mathbb{R}^s\) denotes the parameter space and \(\mathbb{R}^J\) denotes the data space. This scenario presents several challenges. First, the dimensionality of the parameter space often differs from that of the observation data space, leading to an underdetermined system when \(s > J\), where the number of equations is smaller than the number of unknowns. Second, while assuming that \(\mathcal{G}\) maps \(\mathbb{R}^s\) to a proper subset of \(\mathbb{R}^J\) and has a unique inverse as a map from the image of \(\mathcal{G}\) to \(\mathbb{R}^s\), the presence of noise can result in \(y\) not belonging to the image of \(\mathcal{G}\). This issue complicates the inversion of \(\mathcal{G}\) on the data \cite{ghanem2017handbook}. The model in practice is often represented as 
\begin{equation}\label{eq:fowardmodel}
   y = \mathcal{G}(z) + \eta, 
\end{equation}
where \(\eta \in \mathbb{R}^{J}\) signifies the noise. Bayesian inference provides a robust framework for dealing with these complexities and uncertainties.

From a probabilistic perspective, both the data and noise are considered as random variables. We assume the distribution of the noise in advance. The Bayesian solution to the inverse problem is the posterior probability distribution of the unknown parameter \(z\), given the observed data \(y\), denoted as \(p(z|y)\).
The Bayesian approach operates as follows. Given the prior distribution \(\pi_{0}\) of the parameter, the posterior distribution \(\pi\) can be calculated using Bayes' rule as
\begin{displaymath}
    \pi(z) := \frac{1}{I}\exp(-\Psi(z;y))\pi_{0}(z),
\end{displaymath}
where \(I\) denotes the normalization constant 
\begin{displaymath}
    I := \int_{\mathbb{R}^{s}}\exp(-\Psi(z;y))\pi_0(z)dz,
\end{displaymath}
and \(\Psi:\mathbb{R}^s \times \mathbb{R}^{J} \longrightarrow \mathbb{R}\) is the potential, for a given observed data, \(y\). Note that \(\Psi(\cdot,y)\) often refers to the negative log-likelihood. The integration of a function \( f: \mathbb{R}^s \to \mathbb{R} \) with respect to the posterior distribution, is formulated as 
\begin{equation}\label{2}
  \mathbb{E}_\pi [f(Z)] = \frac{\int_{\mathbb{R}^{s}}f(z)\exp(-\Psi(z;y))\pi_0(z)dz}{\int_{\mathbb{R}^{s}}\exp(-\Psi(z;y))\pi_0(z)dz}.
\end{equation}
Monte Carlo (MC) methods and their variants, such as Markov chain Monte Carlo (MCMC), have been extensively utilized for estimating posterior densities. These methods are well-documented in the literature. Recent developments are highlighted in references \cite{bib15ernst2015analysis,bib39schillings2017analysis}. Furthermore, a comprehensive survey of the statistical theory for operator equations in function spaces is available in \cite{bib18gine2021mathematical}. Despite the widespread use of these MC-based methods, they typically offer a root mean-squared error (RMSE) rate of $O(N^{-1/2})$, where $N$ is the sample size. In recent years, there have been significant advancements in developing higher-order numerical methods for computing Bayesian estimates in  Partial Differential Equation-constrained forward problems with distributed uncertain inputs from function spaces. This progress began with the exploration of uniform priors, as documented in \cite{bib40schwab2012sparse}, and has included Gaussian priors. Notable contributions in this field involve the implementation of Smolyak-based quadrature \cite{bib8chen2017hessian} and the multilevel MC and quasi-Monte Carlo (QMC) quadratures \cite{bib38scheichl2017quasi}. These methods, which build upon and extend traditional MC techniques, represent a significant shift toward more sophisticated computational strategies in Bayesian inference. They underscore the ongoing efforts to enhance the efficiency and accuracy of solving inverse problems, especially in complex and high-dimensional data spaces.

While MC methods and their extensions are fundamental in Bayesian inference, they encounter challenges, particularly when dealing with concentrated posterior measures due to large datasets or small noise levels \cite{bib1}. Typically, the noise in the forward model \eqref{eq:fowardmodel} follows $\eta \sim N(0,\Gamma_n)$, with its variance matrix $\Gamma_n$ decaying as a result of incorporating a scaled noise covariance, $\Gamma_n= n^{-1}\Gamma$. 

Due to the concentration effect of the posterior measure, plain MC sampling based on the prior distribution may become inefficient for large $n$. A larger $n$ indicates low intensity noise in the problem, which leads to reduced uncertainty in the model. Hence, the demand for efficient numerical integration algorithms suitable for low-noise scenarios is highly relevant. The introduction of importance sampling (IS) marks a significant advance in addressing these challenges. As a variance reduction technique, IS has been widely studied and implemented in Bayesian Inverse Problems (BIPs). It effectively reweights the sampling process to focus on the more critical regions of the parameter space. This significantly improves the efficiency and accuracy of MC methods under the concentration effect, as detailed in  \cite{bib1}. The integration of IS with traditional MC methods signifies a notable advance in the development of computational methods for Bayesian inference.

Beyond the advancements achieved through IS, the QMC method emerges as a powerful alternative to traditional MC methods. By employing low discrepancy such as lattice point sets or scrambled nets, QMC methods offer deterministic counterparts to MC methods and are capable of achieving faster convergence rates \cite{caflisch1998monte,niederreiter1992random}. Randomized quasi-Monte Carlo (RQMC) provides an unbiased estimator, and its empirical variance can be easily utilized to evaluate the efficiency of RQMC. However, combining RQMC with IS does not always lead to enhanced outcomes, as demonstrated by He et al. \cite{bib2}. 

This paper aims to explore the integration of using the randomly shifted rank-1 lattice rule, a widely used RQMC method, with IS. Recently, Wang and Zheng \cite{wang2024randomly} analyzed the convergence rate of randomly shifted rank-1 lattice rule combined with IS in some statistical and financial models. This paper particularly focuses on challenges posed by concentrated posterior measures. We employ the lattice rule due to its solid theoretical foundation in the reproducing kernel Hilbert space (RKHS), which enables the analysis of the convergence order in relation to noise levels. This research is driven by the need to understand and improve the convergence rates of lattice rule under various IS proposals and their relation to noise levels. Our goal is to provide guidelines for selecting IS densities that enhance the efficiency and accuracy of numerical integration in BIPs settings. In this paper, we aim to develop methods that are robust against the concentration effect and effective in handling large datasets.

Recent years have seen developments in the application of IS and QMC in BIPs. Scheichl et al. \cite{bib38scheichl2017quasi} provided a convergence and complexity analysis of the randomly shifted lattice rule for computing posterior expectations in elliptic inverse problems, employing the prior distribution as the sampling proposal. Herrmann et al. \cite{herrmann2021quasi} focused on Besov priors for admissible uncertain inputs and established conditions for achieving dimension-independent convergence rates. However, these papers did not investigate the effect of noise levels nor the general proposals for IS. Our work relates closely to that of Schillings et al. \cite{bib39schillings2017analysis}, who analyzed the effects of noise levels on MC methods combined with Laplace-based and prior-based IS in BIPs. QMC integration based on the Laplace approximation of the posterior was also studied in \cite{bib39schillings2017analysis}. Unlike \cite{bib39schillings2017analysis}, we study the error bounds of QMC with more general IS proposals, such as Gaussian or $t$-distribution proposals, including Laplace-based and prior-based IS as special cases. We demonstrate that QMC with a proper IS proposal is robust with respect to the concentration of the posterior, achieving an optimal error rate close to $O(N^{-1})$ for a sample size of $N$. Additionally, we provide strategies for selecting such proposals to reclaim the efficiency of QMC.

The rest of this paper is organized as follows. Section~\ref{sec3} provides the background on MC and IS. Section \ref{sec2} presents the foundational theory of the lattice rule. Section \ref{sec:gaussian} discusses the convergence rates of the lattice rule under various IS densities based on Gaussian proposals within the BIP framework and explores their asymptotic properties concerning the noise levels. Results for IS using $t$-distribution proposals are provided in Section~\ref{sec:tdist}. Experimental results are presented in Section \ref{sec4}. Section~\ref{sec:conclusion} concludes the paper.

\section{Problem formulation}\label{sec3}

In this paper, we focus on the concentration effect due to a small noise of the model. We assume that the noise  in the forward model \eqref{eq:fowardmodel}  admits $\eta \sim N(0,\Gamma_n)$ and its variance matrix decays by incorporating a scaled noise covariance $\Gamma_n$ that is inversely proportional to the number of observations $n$ made so far, i.e., $\Gamma_n = n^{-1}\Gamma$. The likelihood of the model \eqref{eq:fowardmodel} is then 
$$p(y|z) = (2\pi n)^{-\frac{J}{2}}|\Gamma|^{-\frac{1}{2}}\exp\left\{-\frac{n}{2}\norm{y-\mathcal{G}(z)}^2_\Gamma\right\},$$
where $\norm{a}^2_\Gamma:=a^\top\Gamma^{-1}a$.
To delve into the impact of noise level, we fix data $y$ and take the negative log-likelihood (omitting a constant term) $\Psi(z;y)=\frac{n}{2}\norm{y-\mathcal{G}(z)}^2_\Gamma=n\Psi(z)$, where $\Psi(z):=\frac{1}{2}\norm{y-\mathcal{G}(z)}^2_\Gamma$. To simplify the notation, the dependence of $\Psi(z)$ on data $y$ has been omitted. The posterior expectation of interest is therefore reformulated as
\begin{equation}\label{3}
   \mathbb{E}_{\pi}[f(z)] = \frac{\int_{\R^s} f(z) \exp(-n \Psi(z)) \pi_0(z) \, dz}{\int_{\R^s} \exp(-n \Psi(z)) \pi_0(z) \, dz} =: \frac{T_1}{T_2},
\end{equation}
where $\pi_0$ and $\pi$ denote the prior and posterior distributions, respectively. This paper primarily focuses on the setting of Gaussian prior $\pi_0 =N(\mu_0, \Sigma_0)$.

To estimate $\mathbb{E}_{\pi}[f(z)]$, we can employ MC methods based on the prior distribution, which allows for direct sampling. However, for large $n$, most of the sample points generated from the prior distribution will be located in the regions of low posterior, making the sampling inefficient.
IS is widely employed to address the issue by using a proper sampling measure denoted by $q(x)$ that reflects the important region of the posterior. By a change of measure, the IS-based estimator for \eqref{3} is given by
\begin{equation}\label{31}
    \frac{\hat T_1}{\hat T_2}:=\frac{\frac{1}{N}\sum_{i=0}^{N-1} f(x_i)\exp(-n\Psi(x_i))w(x_i)}{\frac{1}{N}\sum_{i=0}^{N-1} \exp(-n\Psi(x_i))w(x_i)},
\end{equation}
where $x_i \simiid q(x)$ and $w(x)=\pi_0(x)/q(x)$.
Using the triangle inequality as in \cite{bib38scheichl2017quasi}, the mean squared error (MSE) of the ratio estimator can be bounded via
\begin{equation}\label{eq:boundedratio}
\mbe{\left(\frac{\hat T_1}{\hat T_2}-\frac{T_1}{T_2}\right)^2}\le \frac{2}{T_2^2}\biggl(\mathbb{E}[(\hat{T_1}-T_1)^2]+\mathbb{E}[(\hat{T}_1/\hat{T}_2)^2(\hat{T}_2-T_2)^2]\biggr).
\end{equation}
If $\hat{T}_1/\hat{T}_2$ is bounded, then the MSE of the ratio
estimator is bounded by the MSE of the numerator and denominator estimators in \eqref{31}. Therefore, our analysis will focus on the quadrature error of the numerator estimator $\hat{T}_1$. The analysis for the denominator estimator $\hat{T}_2$ follows straightforwardly by setting $f(z)\equiv 1$ in the numerator estimator.

A common strategy is to choose the proposal $q(x)$ from a specific family of distributions. We first focus on a multivariate Gaussian distribution $N(\mu,\Sigma)$ as a proposal for IS, with density given by
\begin{equation*}
    q(x;\mu,\Sigma) = (2\pi)^{-\frac{s}{2}}|\Sigma|^{-\frac{1}{2}}\exp\left\{-\frac{1}{2}\norm{x-\mu}_\Sigma^2\right\},
\end{equation*}
where $\mu \in \R^{s}$ and $\Sigma \in \R^{s\times s}$. Let $LL^{T} = \Sigma$, so that $X=\mu+LZ\sim N(\mu,\Sigma)$ for $Z\sim N(0,I_s)$. By a change of measure, the numerator in \eqref{3} can be represented by 
\begin{equation}\label{intergran}
   \mathbb{E}_{N(\mu,\Sigma)}[f(X)\exp(-n\Psi(X))\pi_0(X)/q(X;\mu,\Sigma)]=:\mathbb{E}_{N(0,I_s)}[\gis(Z)],
\end{equation}
where 
\begin{equation}\label{eq:gis}
 \gis(z)=f(x)\exp(-n\Psi(x))W(x)  \text{ with }x=\mu+Lz, 
\end{equation}
and the likelihood ratio 
\begin{align}\label{W(z)}
    W(x) &= \frac{\pi_0(x)}{q(x;\mu,\Sigma)}= \frac{q(x;\mu_0,\Sigma_0)}{q(x;\mu,\Sigma)} \notag\\
    &= \frac{|\Sigma|^{1/2}}{|\Sigma_0|^{1/2}}\exp\left\{\frac{1}{2}\norm{x-
    \mu}_{\Sigma}^2 - \frac{1}{2}\norm{x-
    \mu_0}_{\Sigma_0}^2\right\}\notag\\
    &=\frac{|\Sigma|^{1/2}}{|\Sigma_0|^{1/2}}\exp\left\{\frac{1}{2}x^T(\Sigma^{-1}-\Sigma_0^{-1})x +(\mu_0^T\Sigma_0^{-1}-\mu^T\Sigma^{-1})x +\frac{1}{2}(\norm{\mu}_{\Sigma}^2-\norm{\mu_0}_{\Sigma_0}^2)\right\}.
\end{align}

A direct way to choose the proposal is taking $\mu=\mu_0$ and $\Sigma=\Sigma_0$, referred as the prior-based IS (PriorIS), which may be inefficient when the posterior is far away from the prior. A better strategy is to choose a drift parameter $\mu$ that aligns with the mode of the optimal IS density $q_{opt}(z)\propto |f(z)|\exp(-n\Psi(z))\pi_0(z)$. This particular IS density is known as the optimal drift IS (ODIS) \cite{zhang2021efficient}. When $n$ is large, the mode of $q_{opt}(z)$ is close to the maximizer of the likelihood $\exp(-n\Psi(z))$. In the following, we take the drift parameter for ODIS as
\begin{equation}\label{eq:mustar}
\mu^* = \arg\max_{z\in \mbr^s} \exp(-n\Psi(z)) = \arg\min_{z\in \mbr^s} {\Psi(z)}.
\end{equation}
 We assume that $\mu^*$ is the global minimizer of $\Psi(z)$, and $\Psi(\mu^*)=0$, otherwise we can replace $\Psi(z)$ by $\Psi(z)-\Psi(\mu^*)$ in \eqref{3}.
The ODIS density is then given by
\begin{displaymath}
    q_{OD}(x) = q(x;\mu^*,\Sigma_0),
\end{displaymath}
which is independent of the noise level $n$.

Regarding Laplace IS (LapIS), we take the second order Taylor expansion of $\Psi$ at $\mu^*$, i.e,
\begin{align*}
 \Psi(z) \approx \Psi(\mu^*) + \frac{1}{2}(z-\mu^*)^{T}\nabla^2\Psi(\mu^*)(z-\mu^*).
\end{align*}
The LapIS density is then chosen as
\begin{align*}
    q_{Lap}(x) =q(x;\mu^*,n^{-1}\Sigma^*)\propto \exp\left\{-\frac{n}{2}(x-\mu^*)^T\nabla^2 \Psi(\mu^*)(x-\mu^*)\right\},
\end{align*}
where $\Sigma^*=(\nabla^2\Psi(\mu^*))^{-1}$.

\section{Preliminaries on lattice rules}\label{sec2}
In this section, we review the basic properties of the QMC method for estimating \eqref{intergran} and briefly revisit some theories about the RKHS. For more details, we refer to \cite{aronszajn1950theory,le2012multiscale}.

QMC methods are equal-weight quadrature rules for approximating integrals over the unit cube $(0,1)^s$. A necessary step in applying QMC methods to the integral \eqref{intergran} formulated over $\R^{s}$ is to transform the integral into the unit cube $(0,1)^s$. Consider a more general problem
\begin{equation}\label{8}
    I_{s,\phi}(G) := \int_{\R^s} G(z)\phi(z)\,dz = \int_{(0,1)^s}G\circ \Phi^{-1}(\mathtt{u})\, d\mathtt{u},
\end{equation}
where $z = (z_1, \ldots, z_s)^T$ and \( \phi(z) := \prod_{j=1}^{s} \phi(z_j) \) represents the joint density function for \( s \) independent and identically distributed (i.i.d.) random variables with marginal probability density function (PDF) \( \phi(\cdot) \), cumulative distribution function (CDF) \( \Phi(\cdot) \), and quantile function \( \Phi^{-1}(\cdot) \). These functions are applied to each component individually.
 
In this paper, we use a class of RQMC methods called randomly shifted rank-1 lattice rules, where the points are constructed using a generating vector ${z^*}\in \mathbb{N}^{s}$ and a random shift $\Delta\sim U((0,1)^s)$. The estimator is
\begin{equation}\label{lattice_rule}
    \hat{I}_{s}(G\circ\Phi^{-1}) = \frac{1}{N}\sum_{k=0}^{N-1}G\left(\Phi^{-1}\left(\left\lbrace\frac{k{z^*}}{N} + \Delta \right\rbrace\right)\right) ,
\end{equation}
where the brace $\{\cdot\}$ denotes the fractional part of each component.

We should note that the classical Koksma--Hlawka inequality \cite{niederreiter1992random} fails to provide a useful QMC error bound for unbounded function $G\circ \Phi^{-1}$. Most QMC error analyses in the literature are carried out for  a certain space of functions. To handle unbounded functions, we use a weighted RKHS  $\mathcal{F}$ of real valued functions on
$\R^s$ introduced in \cite{kuo2010randomly,kuo2006randomly,bib4}. 
The space contains the functions with mixed first derivatives square integrable under the weight function $\psi$. Let $1{:}s:=\{1,\ldots,s\}$.  For $u\subseteq 1{:}s$, denote $z_u:=(z_j)_{j\in u}$ and $z_{-u}:=(z_j)_{j\in{1:s}\setminus u}$. Moreover, $|u|$ denotes the cardinality of $u$, and $\partial^{u}G(z)$ denotes $\frac{\partial^{|u|}}{\partial z_{u}}G(z)$.  The norm of the space $\mathcal{F}$ is defined by
\begin{equation}\label{rkhs}
   \|G\|_{\mathcal{F}}^{2} = \\ \sum_{u\subseteq{1:s}}\frac{1}{\gamma_{u}}\int_{\R^{|u|}}\bigg(\int_{\R^{s-|u|}}\partial^{u}G(z)\phi_{-u}(z)dz_{-u}\bigg)^2 \psi_u^2(z)dz_{u} , 
\end{equation}
where $\psi_u(z) = \prod_{j\in u}\psi(z_j)$, $\phi_u(z)=\prod_{j\in u}\phi(z_j)$ and their subscripts are omitted when $u=1{:}s$. The weight function $\psi$ reflects the growth properties of the function $G(z)$, and the weight coefficient $\gamma_{u}$ measures the importance of the variable $z_u$ \cite{bib4}. As the weight function decreases more rapidly, the space encompasses functions exhibiting more pronounced growth near the boundary. Moreover, $\mathcal{F}$ is embedded in $L^2_{\phi}(\R^s)$ if for any finite $c$,
\begin{equation}\label{14}
    \int_{-\infty}^{c}\frac{\Phi(t)}{\psi^2(t)}\,dt + \int_{c}^{\infty}\frac{1-\Phi(t)}{\psi^2(t)}\,dt < \infty.
\end{equation} 
Examples of common pairings $(\phi,\psi)$
satisfying \eqref{14} are provided in \cite{kuo2012quasi,bib4}. 
Nichols and Kuo \cite{bib4} showed that  the MSE of the estimator \eqref{lattice_rule} is bounded by
\begin{equation}\label{eq:upperbound}   
 \mathbb{E}^\Delta |I_{s,\phi}(G)-\hat I_{s}(G\circ \Phi^{-1})|^2 \leq e_{s}^{\rm sh}(z^*)^2\|G\|_{\mathcal{F}}^2,
\end{equation}
where 
\begin{equation*}
e_{s}^{\rm sh}(z^*)^2 =\sum_{\emptyset \neq u\subseteq{1:s}}\frac{\gamma_{u}}{N}\sum_{k=0}^{N-1}\prod_{j\in u}\theta\left(\left\{\frac{kz^*_{j}}{N}\right\}\right),
\end{equation*}
and
\begin{equation}\label{eq:theta}
\theta (c)=\int_{\Phi^{-1}(c)}^{\infty} \frac{\Phi(t)-c}{\psi^2(t)}dt + \int_{\Phi^{-1}(1-c)}^{\infty} \frac{\Phi(t)-1+c}{\psi^2(t)}dt - \int_{-\infty}^{\infty} \frac{\Phi^2(t)}{\psi^2(t)}dt. 
\end{equation}

To ascertain the achievable convergence order using the lattice rule, it suffices to examine whether the function $G(z)$ adheres to the conditions stipulated in the RKHS $\mathcal{F}$ and find an optimal generating vector $z^*$ for minimizing the error bound \eqref{eq:upperbound}. However, searching such an optimal generating vector $z^*$ is typically computationally intensive, and conventional search methods can lead to a significant waste of computational resources. The component by component (CBC) algorithm is widely used to find a sub-optimal $z^*$ instead. 
Algorithm~\ref{alg:cbc} presents the process of the CBC algorithm.   Remarkably, \cite{cools2006constructing} proposed fast CBC algorithm, which can substantially enhance the efficiency of searching for ${z^*}$. Next, we introduce the results of error analysis in \cite{bib4}, upon which our subsequent analysis is fundamentally based.

\begin{algorithm}
\caption{CBC Algorithm}\label{alg:cbc}
\begin{algorithmic} 
    \State Set $z^*_1 = 1$
    \For{$d = 2,\dots,s$}
        \State with $z^*_1,\ldots,z^*_{d-1}$ fixed, choose $z^*_d\in 0{:}(N-1)$ and $\gcd(z^*_d,N)=1$ such that $e_{d}^{\rm sh}(z^*_1,\ldots,z^*_{d-1},z^*_d)$ is minimized.
    \EndFor
\end{algorithmic}
\end{algorithm}

\begin{Assumption}\label{assump1}
Assume there exist $r > 1/2$ and $C > 0$ such that  the Fourier series coefficient of $\theta(c)$ given by \eqref{eq:theta}  satisfies
\begin{equation*}
  \hat{\theta}(h):=\frac{1}{\pi^2h^2}\int_{\R}\frac{1}{\psi^2(t)}\sin^2(\pi h \Phi(t))\,dt \leq C|h|^{-2r} \  \ for\ all \ h \neq 0. 
\end{equation*} 
\end{Assumption}
\begin{theorem}\label{thm2.2}
    Let $G \in \mathcal{F}$ with density function $\phi$, weight $\gamma_{u}$ and weight function $\psi$. If Assumption \ref{assump1} is satisfied with a parameter $r>1/2$, then using a randomly shifted lattice rule with $N$ points in $s$ dimensions can be constructed by a CBC algorithm such that, for all $\lambda \in [\frac{1}{2}\ , r)$, the RMSE
\begin{equation}\label{eq:rmsebound}
    \sqrt{E^\Delta |I_{s,\phi}(G)-\hat I_{s}(G\circ \Phi^{-1})|^2}\\
    \leq\bigg(\frac{1}{\varphi(N)}\sum_{\emptyset \ne u\subseteq{1:s}}\gamma_{u}^{\frac{1}{2\lambda}} \left(2C^{\frac{1}{2\lambda}}\zeta\left(\frac{r}{\lambda}\right)\right)^{|u|}\bigg)^{\lambda} \|G\|_{\mathcal{F}},
\end{equation}
where $\zeta(x):=\sum_{k=1}^\infty k^{-x}$ denotes for the Riemann zeta function, and $\varphi(N):=|\{1\le i\le N-1:\gcd(i,N)=1\}|$ denotes the Euler function. 
\end{theorem}
\begin{proof}
    See \cite{bib4} for the proof. 
\end{proof}

\begin{remark}
Note that the Euler's totient function $\varphi(N)$ equals $N-1$ when $N$ is a prime number, and it holds that $\varphi(N) > N/9$ for all $N \leq 10^{30}$ as referenced in \cite{kuo2012quasi}. Consequently, the term $1/\varphi(N)$ can be approximated by a constant factor times $1/N$. This approximation implies an upper bound of $O(N^{-r+\epsilon})$ in \eqref{eq:rmsebound}, where $\epsilon$ can be any arbitrarily small positive number. The parameter $r$, pertaining to commonly encountered pairings of distribution and weight functions, has been computed in the studies such as \cite{bib4}. We select and list some combinations of the density function $\phi$ and the weight function $\psi$ that satisfy \eqref{14} from \cite{kuo2012quasi,bib4} in Table \ref{table1}, which we will reference later in the paper.
\end{remark}

\begin{table}[ht]
    \centering
 \renewcommand{\arraystretch}{1.2} 
     \begin{tabular}{ccc}
        \toprule
          \diagbox{$\psi(t)$}{$\phi(t)$} & $\frac{\exp(-t^2/2\nu)}{\sqrt{2\pi}\nu}$ & $\frac{1}{\sqrt{\nu\pi}}\frac{\Gamma(\frac{\nu+1}{2})}{\Gamma(\frac{\nu}{2})}(1+\frac{t^2}{\nu})^{-\frac{\nu+1}{2}}$ \\\midrule
       $\exp{(-t^2/(2\alpha))}$ &  $r = 1-\frac{\nu}{\alpha}$ & - \\
            & $\alpha>2\nu$ &  \\\hline
        $(1+|t|)^{-\alpha}$& $r = 1-\delta$ & $r=1-\frac{2\alpha+1}{2\nu}$\\
        & $\delta\in(0,\min(\frac{1}{2},\frac{9}{8}\alpha\nu))$ & $2\alpha +1<\nu$ \\
        \bottomrule
    \end{tabular}
    \caption{The parameter $r$ in Assumption~\ref{assump1} for some common combinations of the probability density function $\phi$ and weight function $\psi$.\label{table1}}
    
\end{table}

\section{Results for Gaussian proposals}\label{sec:gaussian}
\subsection{Convergence rates of randomly shifted lattice rule}

Now we provide some conditions such that the function $G(z)$ given in \eqref{eq:gis} belongs to the RKHS under the standard Gaussian density setting in which $\phi(t)=(2\pi)^{-1/2}\exp{(-t^2/2)}$. In this case, we take the weight function as 
\begin{equation}\label{eq:psi1}
    \psi(t) = \exp{(-t^2/(2\alpha))},\ \alpha>2
\end{equation}
as illustrated in Table~\ref{table1}. 
Let $g(\mathtt{u}) =G(\Phi^{-1}(\mathtt{u}))$, where $\Phi$ is the CDF of standard normal  distribution and $\mathtt{u}\in (0,1)^s$. We focus on the so-called growth conditions for (possibly) unbounded function $g(\mathtt{u})$, which were first used in \cite{bib3} for studying scrambled nets variance. To be specific, Owen \cite{bib3} showed that if there exist constants $B_j> 0$ and $B > 0$ such that
\begin{equation}\label{eq:condition1}
|\partial^{v}g(\mathtt{u})| \leq B \prod_{j=1}^s \min(\mathtt{u}_j,1-\mathtt{u}_j)^{-1\{j\in v\}-B_j}
\end{equation}
for any $v \subseteq 1{:}s$, using scrambled net then yields a mean error of $O(N^{-1+\max_j B_j+\epsilon})$ for arbitrarily small $\epsilon>0$. Recently, \cite{bib2} used another growth condition for $G(z)$ on $\mbr^s$ rather than $g(\mathtt{u})$ on the unit cube $(0,1)^s$, i.e.
\begin{equation}\label{eq:condition2}
|\partial^{v}G(z)| \leq B \prod_{j=1}^s (1-\Phi(|z_j|))^{-B_j}.
\end{equation}
He et al. \cite{bib2} proved that \eqref{eq:condition2} implies \eqref{eq:condition1}.

Since for every $t>0$, we have 
\begin{equation*}
\frac{t}{1+t^{2}}\exp\left\{-t^{2}/2\right\} \leq \int_{t}^{\infty}\exp\left\{-\xi^{2}/2\right\}\,d\xi \leq \frac{1}{t}\exp\left\{-t^{2}/2\right\}.
\end{equation*}
Thus 
\begin{equation}
|z_{j}|\exp\left\{z_{j}^2/2\right\}\le(1-\Phi(|z_{j}|))^{-1} \le \left(\frac{1}{|z_{j}|}+|z_{j}|\right)\exp\left\{z_{j}^2/2\right\}.
\end{equation}
As a result, we may use the following form of growth condition instead of \eqref{eq:condition2}
\begin{equation}\label{eq:bdy2}
|\partial^{v}G(z)| \leq B \prod_{j=1}^s \exp\{M_jz_j^2\}
\end{equation}
for some constants $M_j$.
It is easy to see that if \eqref{eq:bdy2} holds, then \eqref{eq:condition2} holds with $B_j=2M_j$. On the other hand, if \eqref{eq:condition2} holds, then \eqref{eq:bdy2} holds with $M_j=B_j/2+\epsilon$ for arbitrarily small $\epsilon>0$. So the two growth conditions \eqref{eq:bdy2} and \eqref{eq:condition2} are considered almost equivalent. For easy of notation, we particularly take all $M_j=M$ in the following assumption. Let $\| z \|^2=\norm{z}_{I_s}^2=z_1^2+\dots+z_s^2$.

\begin{Assumption}\label{assum3}Assume that $G(z)$ is a real-valued function  on $\mbr^s$ and there exist $B>0$ and a growth rate $M\in \R$ such that for any $z\subseteq 1{:}s$,
\begin{equation}\label{eq:condition3}
|\partial^u G (z)| \leq B\exp\left\{M\|z\|^2\right\}.
\end{equation}
\end{Assumption}

\begin{theorem}\label{wuf}
    If $G(z)$ defined on $\R^{s}$ satisfies Assumption \ref{assum3} with $0<M < 1/4$, then $G$ belongs to the RKHS  $\mathcal{F}$ equipped with standard normal density function $\phi$ and weight function $\psi$ defined by \eqref{eq:psi1} with  $\alpha>2$, and  a randomly shifted lattice rule $\hat{I}_{s}(G\circ\Phi^{-1})$ with $N$ points can be constructed by a CBC algorithm with an RMSE of 
    $O(N^{-1+2M+\epsilon})$ for arbitrarily small $\epsilon > 0$.
\end{theorem}
\begin{proof}
We now show that $G(z)$ belongs to the RKHS by examining the boundedness of the norm \eqref{rkhs}, which is a  sum of the terms
\begin{equation*}
    T_u:=\int_{\R^{|u|}}\Bigg(\int_{\R^{s-|u|}}\partial^{u}G(z)\phi_{-u}(z)\,dz_{-u}\Bigg)^{2}\psi_{u}(z)^{2}\,dz_{u}
\end{equation*}
with $u\subseteq 1{:}s$.
Using Assumption \ref{assum3}, we have
\begin{align*}
    T_u&\le C\int_{\R^{|u|}}\Bigg( \int_{\R^{s-|u|}}  \exp\left\{M\|z\|^2\right\}\exp\left\{-\frac{\sum_{j\notin u}z_{j}^{2}}{2}\right\}\,dz_{-u}\Bigg)^{2}\exp\left\{-\frac{\sum_{j\in u}z_j^2}{\alpha}\right\}\,dz_{u} \\
    &= C\int_{\R^{|u|}}\Bigg( \int_{\R^{s-|u|}}  \exp\left\{(M-1/2)\sum_{j\notin u}z_{j}^{2}\right\}\,dz_{-u}\Bigg)^{2}\exp\left\{(2M-1/\alpha)\sum_{j\in u}z_j^2\right\}\,dz_{u}
\end{align*}
for a constant $C>0$.
Since $2<\alpha<1/(2M)$, we have $T_u<\infty$. Applying Theorem \ref{thm2.2} with $\lambda=r-\epsilon/2=1-1/\alpha-\epsilon/2$ and $1/\alpha=2M+\epsilon/2$ completes the proof. 
\end{proof}

If $M$ is arbitrarily small, then we can get an RMSE rate of $O(N^{-1+\epsilon})$ for the randomly shifted lattice rule. This is the case for $G(z) = \exp(\zeta^Tz)$ with any fixed vector $\zeta\in \R^s$. Under the growth condition \eqref{eq:condition2}, He et al. \cite{bib2} showed that scrambled net quadrature rule yields an RMSE of  $O(N^{-1+\max_j B_j+\epsilon})$. Taking $B_j=2M$ gives the same rate of the randomly shifted lattice rule. From this point of view, under the growth condition, the two major classes of RQMC methods enjoy the same RMSE rate for unbounded integrands. 

We are ready to perform the error analysis for  randomly shifted lattice rule with the IS function $\gis(z)=f(x)\exp(-n\Psi(x))W(x)$ given by \eqref{eq:gis}, where $x=\mu+Lz$. Recall that $\Sigma=LL^\top$ is the covariance matrix of Gaussian proposal and $\Sigma_0$ is the covariance matrix of the prior distribution. Let $\lambda_{\min}(E)$ and $\lambda_{\max}(E)$ denote the smallest and largest eigenvalues of the matrix $E$, respectively, and let $\mathbb{N}_0$ be the set of all nonnegative integers.  Due to the transformation $x=\mu+Lz$ in formulating $\gis(z)$, we require higher order mixed derivatives.  For $a = (a_1,...,a_s)\in \mathbb{N}_0^s$, define $|a|=a_1 +\dots + a_s$, $a!=a_1!\dots\alpha_s!$, and $$D^{a}G(z):=\frac{\partial^{|a|}}{\partial z_1^{a_1}\dots \partial z_s^{a_s}}G(z).$$
As discussed, Assumption~\ref{assum3s} addresses higher order mixed derivatives and represents a somewhat stronger condition compared to Assumption~\ref{assum3}. This assumption is crucial for managing the complexities associated with these derivatives. Define $G_0(z)=f(z)\exp(-n\Psi(z))$.

\begin{Assumption}\label{assum3s}Suppose $G(z)$ is a real-valued function on $\mbr^s$ and there exist $B>0$ and a growth rate $M\in\R$ such that for any $a\in\mathbb{N}_0^s$ satisfying $|a|\le s$,
\begin{equation}\label{eq:condition33}
|D^a G(z)| \leq B\exp\left\{M\|z\|^2\right\}.
\end{equation}
\end{Assumption}

\begin{theorem}\label{ISqmc}
Let $\gis(z)=G_0(\mu+Lz)W(\mu+Lz)$ with $W(x)$ given by \eqref{W(z)} and $\Sigma=LL^T$. Suppose that $G_0(z)$ satisfies Assumption~\ref{assum3s} with the growth rate $M\in\R$. If 
\begin{equation}\label{eq:gammacond}
   \gamma := \lambda_{\min}(\Sigma)\left(\frac{1}{\lambda_{\max}(\Sigma_0)}-2M\right) > 1/2, 
\end{equation}
then a randomly shifted lattice rule $\hat{I}_{s}(\gis\circ\Phi^{-1})$ with $N$ points can be constructed by a CBC algorithm with an RMSE of  $O(N^{-\min(\gamma,1)+\epsilon})$ for arbitrary small $\epsilon > 0$.
\end{theorem}
\begin{proof}
Let $x = \mu+Lz$ so that $\gis(z)=G_0(x)W(x)$.    By the Faa di Bruno formula, due to the specific form of $W(x)$ given by \eqref{W(z)}, $\partial^u \gis(z)$ is a sum of finite terms of the form
\begin{align*}
    T(z) = \left(\prod_{i=1}^s z_i^{t_i}\right)D^a G_0(x)W(x)
\end{align*}
for $a\in\mbn_0^s$ satisfying $|a|\le s$ and some integers $0\le t_i\le s$.
By Assumption~\ref{assum3s} and using \eqref{W(z)}, we have
\begin{align*}
    |T(z)| &\lesssim\left(\prod_{i=1}^s|z_i|^{t_i}\right)\exp\left\{M\|x\|^2 + \frac{1}{2}\norm{x-\mu}_{\Sigma}^2 - \frac{1}{2}\norm{x-\mu_0}_{\Sigma_0}^2\right\},
\end{align*}
where the symbol $\lesssim$ is used for hiding a constant independently of $z$.
As a result, due to $\alpha>2$,
\begin{align}
    S &:= \int_{\R^u}\left(\int_{\R^{-u}}T(z)^2\phi_{-u}(z)dz_{-u}\right)\psi_u(z)^2dz_u\\\notag
    &\lesssim \int_{\R^s}\left(\prod_{i=1}^s|z_i|^{2t_i}\right)\exp\left\{2M\|x\|^2 + \norm{x-\mu}_{\Sigma}^2 - \norm{x-\mu_0}_{\Sigma_0}^2\right\}\psi(z)^2dz\\\notag
    & \lesssim \int_{\R^s}\left(\prod_{i=1}^s|z_i|^{2t_i}\right)\exp\left\{2M\|x\|^2 + (1-1/\alpha)\norm{x-\mu}_{\Sigma}^2 - \norm{x-\mu_0}_{\Sigma_0}^2\right\}dx.\notag
\end{align}
If $2MI_s+(1-1/\alpha)\Sigma^{-1}  - \Sigma_0^{-1}$ is negative definite, one gets $S<\infty$ and then we have $\|\gis\|_{\mathcal{F}}< \infty$ by noticing each term in \eqref{rkhs} is bounded. To this end, it suffices to take
\begin{align*}
    2M+\frac{1-1/\alpha}{\lambda_{\min}(\Sigma)}-\frac 1{\lambda_{\max}(\Sigma_0)}< 0.
\end{align*}
Therefore, $2<\alpha<1/(1-\gamma)^+$. Applying Theorem \ref{thm2.2} with $\lambda=r-\epsilon/2=1-1/\alpha-\epsilon/2$ and $1/\alpha=(1-\gamma)^++\epsilon/2$ we can complete the proof, where $t^+:=\max(t,0)$.
\end{proof}

\begin{corollary}\label{cor:gaussian}Consider the same setting as in Theorem~\ref{ISqmc}. Let  $\epsilon>0$ be an arbitrary small constant.
\begin{itemize}
    \item If $\Sigma=\Sigma_0$ (as in PriorIS and ODIS) and $M<1/(4\lambda_{\max}(\Sigma_0))$, the random shifted lattice rule has an RMSE of $O(N^{-1+2M^+\lambda_{\max}(\Sigma_0)+\epsilon})$.
    \item If $\Sigma=n^{-1}\Sigma^*$ (as in LapIS) and 
    \begin{equation}\label{eq:gamcond2}
        \gamma_n=\frac{\lambda_{\min}(\Sigma^*)}{n}\left(\frac1{\lambda_{\max}(\Sigma_0)}-2M\right)>1/2,
    \end{equation}
    where $\Sigma^*= (\nabla^2\Psi(\mu^*))^{-1}$, we can get a RMSE of $O(N^{-\min(\gamma_n,1)+\epsilon})$. 
\end{itemize}
\end{corollary}
\begin{proof}
If $\Sigma=\Sigma_0$, $2MI_s+(1-1/\alpha)\Sigma^{-1}  - \Sigma_0^{-1}=2MI_s-(1/\alpha)\Sigma_0^{-1}$ needs to be negative definite. We thus take $2M-1/(\lambda_{\max}(\Sigma_0)\alpha)<0$. This implies $2<\alpha<1/(2M^+\lambda_{\max}(\Sigma_0))$.
The second part can be proved by applying Theorem \ref{ISqmc} with  $\lambda_{\min}(\Sigma)=\lambda_{\min}(\Sigma^*)/n$.
\end{proof}

\begin{remark}\label{rem:qmcrate}
Although the case of $M>0$ is of interest, Theorem~\ref{ISqmc} and  Corollary~\ref{cor:gaussian} include the case of $M\le 0$ in which  the function $G_0(z)$ and its partial derivatives are bounded. It is important to note that a larger value of \( M \) encompasses a broader range of functions. However, as implied in the condition \eqref{eq:gammacond}, the growth rate $M$ should be bounded above by the constant $1/(2\lambda_{\max}(\Sigma_0))$.  If $M$ is an arbitrarily small positive number and $\lambda_{\min}(\Sigma)\ge \lambda_{\max}(\Sigma_0)$, the lattice rule constructed according to Theorem 3.2 rule yields an RMSE of $O(N^{-1+\epsilon})$ as suggested by Theorem~\ref{ISqmc}. He et al. \cite{bib2} obtained the same rate for scrambled net based IS. It should be noted that unbounded integrands with an arbitrarily small growth rate $M$ is friendly to QMC. We refer this situation as the \textit{QMC-friendly growth condition} for unbounded integrands, under which the ratio $\lambda_{\min}(\Sigma)/\lambda_{\max}(\Sigma_0)$ has an impact on the RMSE rate of IS. Corollary~\ref{cor:gaussian} shows that PriorIS and ODIS enjoy RMSE of $O(N^{-1+\epsilon})$  if $M$ is an arbitrarily small. On the contrary, the RMSE rate of LapIS depends on the noise level $n$. The rate deteriorates as $n$ goes up. Moreover, the condition \eqref{eq:gamcond2} may fail to hold for large enough $n$.
\end{remark}

\subsection{Discussion on the boundary condition for practical BIPs}
Theorem~\ref{ISqmc} requires that the function $G_0(z)=f(z)\exp(-n\Psi(z))$ satisfies the boundary condition in Assumption~\ref{assum3s}. Note that
\begin{equation}\label{eq:g0}
    D^a G_0(z) = \sum_{b+c=a} D^{b}
    f(z)D^{c}h_n(z),
\end{equation}
where $a,b,c\in \mbn_0^s$ and $$h_n(z)=\exp(-n\Psi(z))=\exp\left\{-\frac{n}{2}\norm{y-\mathcal{G}(z)}_\Gamma^2\right\}.$$
We now assume that $f(z)$ satisfies Assumption~\ref{assum3s} with the growth rate $M=M_f$.
This is the case of estimating moments ($f(z)=z_i^k$), moment generating function ($f(z)=\exp(t^Tz)$), characteristic function ($f(z)=\exp(it^Tz)$) of the posterior distribution with an arbitrarily small $M_f>0$. It remains to check whether $h_n(z)$ satisfies the boundary condition in Assumption~\ref{assum3s} for practical BIPs.

First of all, if the response operator $\mathcal{G}$ is linear, say $\mathcal{G}(z)= Az$ for a matrix $A$, then it is easy to see that $h_n(z)=\exp\{-(n/2)\norm{y-Az}_\Gamma^2\}$ has bounded partial derivatives. In this case, by \eqref{eq:g0}, $G_0(z)$ satisfies the boundary condition with the same rate $M_f$ for $f(z)$. However, this case is of less interest since the posterior is exactly Gaussian.

We next consider a more general case of $\mathcal{G}$ that includes the linear operator as a special case. Recall that $\mu^*$ given by \eqref{eq:mustar} is the minimizer of $\Psi(z)$ in $\mbr^s$. By taking the second order Taylor polynomial of $\Psi(z)$ around $\mu^*$ arrives at
\begin{equation}\label{eq:remainer}
    \Psi(z) =  \frac{1}{2}||z-\mu^*||_{\Sigma^*}^2 +R(z),
\end{equation}
where $R(z)$ is the remainder,  $\Sigma^*=(\nabla^2\Psi(\mu^*))^{-1}$, and we use the default setting of $\Psi(\mu^*)=0$. We  assume that  $\Psi(z)$ has a lower bound as shown below. 

\begin{Assumption}\label{5}
    Assume that there exists a constant $\delta\in [0,1]$ such that 
    $$\Psi(z) \ge \frac{\delta}{2}\norm{z-\mu^*}^2_{\Sigma^*}\text{ for all }z\in\mbr^s.$$
\end{Assumption}

Assumption~\ref{5} holds trivially when $\delta=0$ since $\Psi(z)=(1/2)\norm{y-\mathcal{G}(z)}_\Gamma^2\ge0$. The case of $\delta=0$ is included in Assumption~\ref{5} for studying more general models in the following.

\begin{theorem}\label{thm:g0}
    Suppose that $f(z)$ and $\Psi(z)$ satisfy  Assumption~\ref{assum3s} with  $M=M_f$ and  $M=M_\Psi$, respectively. If Assumption~\ref{5} holds, then  Assumption~\ref{assum3s} holds for $G_0(z)=f(z)\exp(-n\Psi(z))$ with any growth rate $M>sM_\Psi+M_f-\delta n/(2\lambda_{\max}(\Sigma^*))$.
\end{theorem}
\begin{proof}
 By the Faa di Bruno formula, we have 
\begin{equation}\label{eq:dhn}
    D^a h_n(z) = \exp(-n\Psi(z))\sum_{P\in \Pi(a)}(-n)^{|P|} \prod_{\beta\in  P}(D^{\beta}\Psi)(z),
\end{equation}
where $\Pi(a)$ denotes the set of all $P=(a^{(1)},\dots,a^{(k)})$ satisfying $\sum_{i=1}^ka^{(i)}=a$ and $0\neq a^{(i)}\in \mathbb{N}_0^s$, and $|P|:=k\le |a|\le s$. Under the conditions in the theorem, we have
\begin{align*}
    |D^{a} h_n(z)| &\lesssim  n^s\exp\left\{-\frac{\delta n}{2}||z-\mu^*||_{\Sigma^*}^2+sM_\Psi||z||^2\right\}.
\end{align*}
 By \eqref{eq:g0} and $|D^a f(z)|=O(\exp(M_f||z||^2))$ for $|a|\le s$, we have 
$$|D^a G_0(z)|\lesssim \exp\left\{-\frac{\delta n}{2}||z-\mu^*||_{\Sigma^*}^2+(sM_\Psi+M_f)||z||^2\right\}\lesssim\exp\{M\norm{z}^2\},$$
for any $M>sM_\Psi+M_f-\delta n\lambda_{\min}((\Sigma^*)^{-1})/2=sM_\Psi+M_f-\delta n/(2\lambda_{\max}(\Sigma^*))$.
\end{proof}


Assumption \ref{5} draws inspiration from \cite{helin2022non}, which establishes an upper limit on the unnormalized log-posterior density, in contrast to focusing on the log-likelihood. In this context, the constant $\delta$ acts as a scaling factor and can be chosen to be arbitrarily small. With $\delta > 0$, Assumption \ref{5} restricts our analysis to likelihoods whose tail decay is not slower than that of a Gaussian distribution. For the linear operator $\mathcal{G}$, Assumption \ref{5} holds trivially, while under the nonlinear case, this assumption with 
$\delta>0$  also holds under certain conditions on the response operator $\mathcal{G}$.

\begin{Example}\label{exam:nonlinear}
  The response operator $\mathcal{G}$ is given by a linear mapping with a small nonlinear perturbation $\mathcal{G}(z) = Az + \tau F(z)$, where $A\in \R^{s\times s}$, $F(z) = (F_1(z),\dots,F_s(z))^T\in \R^s$, and $\tau\ge 0$. This example was studied in \cite{helin2022non}. For this case, $\Psi(z) =\frac 12 \norm{Az + \tau F(z)-y}_{\Gamma}^2$. Following the proof of  Proposition 5.6 of \cite{helin2022non}, we can verify Assumption \ref{5} with  $\delta>0$ if there exist $\tau_0,C_1,C_2>0$  such that 
$$\sup\left\{\left\lvert\sum_{i=1}^s \partial_{z_i} F_k(z)x_i\right\rvert:\norm{x}_{\Sigma^*}\le 1\right\}\le C_1,$$
and
$$\sup\left\{\left\lvert\sum_{i=1}^s \sum_{j=1}^s\partial_{z_iz_j} F_k(z)x_iy_j\right\rvert:\norm{x}_{\Sigma^*}\le 1,\norm{y}_{\Sigma^*}\le 1\right\}\le C_2,$$
for all $k=1,\dots,s$, $z\in \R^s$ and $\tau\in[0,\tau_0]$. It is easy to see that if all $F_k(z)$ satisfy Assumption~\ref{assum3s} with a growth rate $M_F$, then $\Psi(z)$ satisfies Assumption~\ref{assum3s} with the growth rate $M_\Psi = 2M_F$. The conditions in Theorem~\ref{thm:g0} are therefore satisfied.
\end{Example}

\begin{Example}
Consider the model inverse problem of determining the distribution of the random diffusion coefficient of a divergence form elliptic PDEs from observations of a finite set of noisy continuous functionals of the solution. The forward problem is an elliptic PDEs given by
\begin{equation}\label{eq:EPDE}
    \begin{aligned}
    -\nabla_x\cdot (a(x,\omega)\nabla_x p(x,\omega)) &= g(x)\ \text{in }D,\\
    p(x,\omega) &= 0 \text{ on }\partial D,
\end{aligned}
\end{equation}
with almost every $\omega\in\Omega$ in a complete probability space $(\Omega,\mathcal{F},\mathbb{P})$,
where $D\subseteq \R^d$ $(d=1,2,3)$ is a bounded Lipschitz domain, $\partial D$ is its boundary, and $a(x,\omega)$ is the diffusion coefficient.
We focus on lognormal random coefficients with finite-dimensional parameter vectors
\begin{equation}
   a(x,\omega)= a_s(x,z) := a_*(x) + a_0(x)\exp\left\{\sum_{j=1}^s z_j \xi_j(x)\right\},
\end{equation}
where $a_*(x)\ge 0$, $a_0(x)>0$ for all $x\in D$,  $b_j=\norm{\xi_j(x)}_{L^\infty(\Bar{D})}<\infty$, and $z_j\sim N(0,1)$ independently.
We study the problem \eqref{eq:EPDE} in its weak form: seeking a solution $p(\cdot,z)\in V:=H_0^1(D)$ such that 
$$\int_D a_s(x,z)\nabla_x p(x,z)\cdot \nabla_x v(x)d x=\int_D g(x)\cdot v(x)d x$$
for all $v\in V$ and almost $z\in \Omega$.
The response operator is given by $\mathcal{G}(z)=\mathcal{H}(p(\cdot,z))\in \R^J$ for an observation operator $\mathcal{H}:V\to \R^J$.
Our goal is to estimate the posterior expectation of $f(z) = \mathcal{T}(p(\cdot,z))$, where $p(x,z)$ is the  solution of \eqref{eq:EPDE} and $\mathcal{T}\in V'$ the dual space of $V$. Scheichl et al. \cite{bib38scheichl2017quasi} studied the error rate of randomly shifted lattice rule for this BIP in which the prior is served as the proposal (i.e., PriorIS). However, that work  considers neither other proposals nor the effect of noise level. In this paper, we are able to provide an RMSE rate for general Gaussian proposals by verifying   Assumption~\ref{assum3s} for $f(z)$ and $\Psi(z)$. Due to the linearity of $\mathcal{T}$ in formulating $f(z)$, using the following result established in \cite{graham:2015} 
$$\norm{D^\beta p(\cdot,z)}_V\le \frac{|\beta|!}{(\ln 2)^{|\beta|}}\left(\prod_{j=1}^s b_j^{\beta_j}\right)\frac{\norm{g}_{V'}}{a_{\min}(z)},$$
where $a_{\min}(z)=\min_{x\in\bar D}a_s(x,z)$ and $\beta\in \mathbb{N}_0^s$, we have
\begin{align*}
|D^\beta f(z)|&\le \norm{\mathcal{T}}_{V'}\norm{D^\beta p(\cdot,z)}_V\\&\le \frac{|\beta|!}{(\ln 2)^{|\beta|}}\left(\prod_{j=1}^s b_j^{\beta_j}\right)\frac{\norm{\mathcal{T}}_{V'}\norm{g}_{V'}}{a_{\min}(z)}\\&\lesssim 1/a_{\min}(z)\lesssim \prod_{j=1}^s \exp(b_j|z_j|),
\end{align*}
where the last inequality comes from the fact that $\min_{x\in \bar D}\xi_j(x)\ge -b_j$. Therefore, $|D^\beta f(z)|=O(\exp\{M\norm{z}^2\})$ for any $M>0$. Now consider $\Psi(z) = \frac 1 2 \norm{y-\mathcal{H}(p(\cdot,z))}_\Gamma^2$. Note that $\Psi(z)$ is not a linear function of $p$. However, $D^\beta \Psi(z)$ is a linear combination of terms of the form $D^{\beta^{(1)}}\mathcal{H}_i(p(\cdot,z))D^{\beta^{(2)}}\mathcal{H}_j(p(\cdot,z))$ for $i,j=1,\dots,J$. As a result,  $|D^\beta \Psi(z)|\lesssim \prod_{j=1}^s \exp(b'_j|z_j|)$ for some $b'_j$ if all $\mathcal{H}_i$ are linear functions of $p$. Assumption~\ref{assum3s} thus holds for $f(z)$ and $\Psi(z)$ with any $M>0$. For this example, we are not able to verify Assumption~\ref{5} with $\delta>0$. To apply our results, we can take $\delta=0$. By Theorem~\ref{thm:g0}, $G_0(z)$ satisfies Assumption~\ref{assum3s} with any $M>0$. Together with Corollary~\ref{cor:gaussian}, ODIS and PriorIS yields an RMSE of $O(N^{-1+\epsilon})$. But it fails to provide an error rate for LapIS since we take $\delta=0$ for this model (see Remark~\ref{rem:qmcrate} for some discussions). It would be interesting to work out a $\delta>0$ or relax Assumption~\ref{5}. In the next subsection, we work with a weaker condition compared to Assumption~\ref{5}. As shown in  \cite{graham:2015}, these results also hold for finite element solutions of \eqref{eq:EPDE}. Moreover,  \cite{graham:2015} showed the quadrature error decays with $O(N^{-1+\epsilon})$ with the implied constant  independent of the dimension $s$ by choosing proper weight parameters $\gamma_u$.
\end{Example}

\subsection{Concentration effects on the RKHS norm}
In this subsection, we assess the stability of the norm $\|\cdot\|_{\mathcal{F}}$ with respect to the noise level $n$. We distinguish among different importance densities and among scenarios with finite and infinite $n$. As $n$ increases, employing the Gaussian family for IS may result in functions that fall outside the RKHS. Consequently, this necessitates a reassessment of the asymptotic properties of the convergence rate with respect to $n$.

Prior to our analysis, we establish certain assumptions regarding the function $\Psi(z)$ to ensure analytical tractability. For the asymptotic analysis of integration, we utilize the classical Laplace method, as detailed in Wong's work \cite{wong2001asymptotic}. This method has been previously applied in \cite{bib1} for the investigation of Laplace-based MC methods in BIPs.

\begin{Assumption}\label{Convex}
Denote
\begin{equation}\label{eq:jn}
    J_n= \int_{\R^s} Q(z)\exp{(-n F(z))dz}.
\end{equation}
Assume that
\begin{enumerate}
    \item $F(z)$ has a global minimizer $c^*\in \mbr^s$ and   $\nabla^2F(c^*)$ is positive definite,
    \item the function $Q(z)$ is $2p+2$ times continuously differentiable and $F(z)$ is $2p+3$ times continuously differentiable in the neighborhood of $c^*$ for a $p\geq 0$, and
    \item  the integral $J_n$ given by \eqref{eq:jn} converges absolutely for each $n \in \mathbb{N}$.
\end{enumerate}
\end{Assumption}

\begin{theorem}\label{thm1}
If Assumption \ref{Convex} holds,
then we have 
\begin{equation*}
    J_n= \exp\{-nF(c^*)\}n^{-\frac{s}{2}}\bigg(\sum_{k=0}^{p}c_{k}(Q)n^{-k} + O(n^{-p-1}) \bigg),
\end{equation*}
where
\begin{equation}\label{coe}
    c_k(Q) = \sum_{a \in \mathbb{N}_0^s:|a|=2k} \frac{\kappa_{a}}{a!}D^a H(0)\in \R,
\end{equation}
$\kappa_{a}\in \R$, 
$H(z):= Q(h(z))\det(\nabla h(z))$ with $h:\omega \longrightarrow U(c^*)$ is a diffeomorphism between $0\in\omega$ and a neighborhood $U(c^*)$ of the  minimizer $c^*$ mapping $h(0)=c^*$ with $\det(\nabla h(0))=1$. When $Q(z)=q_0(z)\prod_{i=1}^Kq_i(z)$ with $q_i(c^*)=0$ for all $i=1,\dots,K$,  $Q(z)$ is said to possess a zero of order $K$ under which $c_k(Q)=0$ for $k=0,\dots,\lfloor K/2\rfloor$ so that $J_n\sim \exp\{-nF(c^*)\}n^{-s/2-\lfloor K/2\rfloor}$.
\end{theorem}
\begin{proof}
    See Theorem 1 of \cite{bib1}.
\end{proof}

\begin{theorem}\label{thm9}
If $\mu$ and $\Sigma$ used in IS do not depend on $n$ and Assumption \ref{Convex} holds with $F(z)=2\Psi(z)$ and $$Q(z)=Q_1(z)=\exp\{-\norm{z-\mu}_{\Sigma}^2/\alpha\}\left(D^c g_1(z)\prod_{ \beta\in  P}D^{\beta}\Psi(z)\right)^2$$ with $g_1(z)=f(z)W(z)$ for all $c\in\mbn_0^s$ with $|c|\le s$ and  $P=(a^{(1)},\dots,a^{(k)})$ satisfying $0\neq a^{(i)}\in \mathbb{N}_0^s$ and $\sum_{i=1}^k|a^{(i)}|\le s$, then we have
$$\|\gis(z)\|_{\mathcal{F}}=O(n^{s/4}).$$
If $\mu=\mu^*$ and $\Sigma=(m/n)\Sigma^*$ for a constant $m>0$ and Assumption \ref{Convex} holds with $F(z)=2\Psi(z)-(1/m-1/\alpha)\norm{z-\mu^*}_{\Sigma^*}^2$ whose global minimizer is $\mu^*$, and $$Q(z)=Q_2(z)=\left(D^c g_2(z)\prod_{ \beta\in  P}D^{\beta}\Psi_1(x)\right)^2$$ with $g_2(x)=f(x)\pi_0(x)$ and $\Psi_1(x)=\Psi(x)-\frac 1{2m}\norm{x-\mu^*}_{\Sigma^*}^2$ for all possible $c$ and $P$ as before, then we have $$\|\gis(z)\|_{\mathcal{F}}=O(n^{-s/2}).$$
\end{theorem}

\begin{proof}
 First, we consider the case that $\mu$ and $\Sigma$ do not depend on the noise level $n$.
Let $$\ggis(x)=f(x)W(x)\exp(-n\Psi(x))=g_1(x)h_n(x),$$ where $g_1(x)=f(x)W(x)$ does not depend on $n$ and again $h_n(x)=\exp(-n\Psi(x))$. Note that $\gis(z)=\ggis(x)$ with $x=\mu+Lz$. Denote $e_i\in \mathbb{N}_0^s$ be the vector  with the $i$th entry one and zero otherwise. Let $\ell=(\ell_1,\dots,\ell_s)\in \{1{:}s\}^s$.
By the Faa di Bruno formula, we have
$$\partial^u \gis(z)=\sum_{\ell_u\in \{1{:}s\}^{|u|}}\left(\prod_{i\in u}D^{e_{\ell_i}}\ggis\right)(x)\prod_{i\in u}L_{\ell_i,i}.$$
Let $\norm{L}_{\max}=\max_{i,j=1,\dots,s}|L_{ij}|$. Then
$$|\partial^u \gis(z)|\le \norm{L}_{\max}^{|u|}\sum_{\ell_u\in\{1{:}s\}^{|u|}}|D^{a(\ell_u)}\ggis(x)|,$$
where $a(\ell_u) = \sum_{i\in u}e_{\ell_i}$ satisfying $|a(\ell_u)|=|u|\le s$.
By \eqref{eq:dhn}, we have
\begin{align}
D^{a(\ell_u)}\ggis(x)&=\sum_{b+c=a(\ell_u)}D^b h_n(x)D^c g_1(x)\\
&=h_n(x)\sum_{b+c=a(\ell_u) \atop P\in \Pi(b)}(-n)^{|P|}D^c g_1(x)\prod_{ \beta\in  P}(D^{\beta}\Psi)(x).
\end{align}

Recall that the RKHS norm is  
\begin{displaymath}
    \begin{aligned}
\|\gis(z)\|_{\mathcal{F}}^2=\sum_{u\subseteq 1{:}s}\frac{F_u(n)}{\gamma_u}.
    \end{aligned}
\end{displaymath}
Among the summation, 
\begin{align}\label{eq:fun1}
F_u(n)&=\int_{\R^u}\left(\int_{\R^{s-|u|}}\partial^u \gis(z)\phi_{-u}(z)dz_{-u}\right)^2\psi_{u}(z)^2dz_u\\\notag
&\le\int_{\R^s}\left(\partial^u \gis(z)\right)^2\psi(z)^2dz\\\notag
&\lesssim\norm{L}_{\max}^{2|u|} \sum_{\ell_u\in \{1{:}s\}^{|u|}} \sum_{b+c=a(\ell_u) \atop P\in \Pi(b)} n^{2|P|}\int_{\R^s}h_n^2(x)\left(D^c g_1(x)\prod_{ \beta\in  P}(D^{\beta}\Psi)(x)\right)^2\psi(z)^2dx\\\notag
&= \frac{\norm{L}_{\max}^{2|u|}}{|\Sigma|^{1/2}} \sum_{\ell_u\in \{1{:}s\}^{|u|}} \sum_{b+c=a(\ell_u) \atop P\in \Pi(b)} n^{2|P|}J_{n},
\end{align}
where the symbol $\lesssim$ is used to hide the constant that does not depend on the noise level $n$, and we use a change of variable $x=\mu+Lz$ and $\psi(z)=\exp\{-\norm{z}^2/(2\alpha)\}$ to obtain
\begin{equation}
    J_{n} = \int_{\R^s}\exp\{-2n\Psi(x)-\norm{x-\mu}_{\Sigma}^2/\alpha\}\left(D^c g_1(x)\prod_{ \beta\in  P}(D^{\beta}\Psi)(x)\right)^2dx.
\end{equation}
Let $$Q_1(x) = \exp\{-\norm{x-\mu}_{\Sigma}^2/\alpha\}\left(D^c g_1(x)\prod_{ \beta\in  P}(D^{\beta}\Psi)(x)\right)^2,$$ and let $|P|_1:=|\{\beta\in P:|\beta|=1\}|$ be the number of elements with $|\beta|=1$ and $\beta\in P$. Taking the case of $b=e_1+e_2+2e_3$ for instant, if $P = \{e_1,e_2,2e_3\}$, then $|P|_1=2$, and if $P=\{e_1,e_2,e_3,e_3\}$, we have $|P|=|P|_1=|b|=4$.
It then follows  
\begin{align*}
n^{2|P|}J_{n} &=n^{2|P|}\int_{\R^s}\exp\{-2n\Psi(x)\}Q_1(x)dx\\
&\lesssim n^{-s/2+2|P|-|P|_1
}\\
&\le n^{-s/2+\max_{P\in\Pi(b)}(2|P|-|P|_1
)}\\
&=n^{|b|-s/2},
\end{align*}
where we use the fact that  $Q_1(x)$ possesses a zero of order $2|P|$ leading to $$J_n=\int_{\R^s}\exp\{-2n\Psi(x)\}Q_1(x)dx=O(  n^{-s/2-|P|_1})$$
by Theorem \ref{thm1} with $c^*=\mu^*$ and $F(\mu^*)=2\Psi(\mu^*)=0$, and $\max_{P\in\Pi(b)}(2|P|-|P|_1)=|b|$. We thus arrive at $\|\gis(z)\|_{\mathcal{F}}^2=O(n^{s/2})$. 

Next consider the case $\mu=\mu^*, \ \Sigma=(m/n)\Sigma^*$ that depends on the noise level $n$ linearly. For this case, $\norm{L}_{\max}=O(n^{-1/2})$ and $|\Sigma|=O(n^{-s})$. Now we rewrite that $$\ggis(x)=(2\pi)^{s/2}|\Sigma|^{1/2}g_2(x)\tilde h_n(x),$$ where $g_2(x)=f(x)\pi_0(x)$ does not depend on $n$ and $\tilde h_n(x)=\exp\{-n(\Psi(x)-\frac 1{2m}\norm{x-\mu^*}_{\Sigma^*}^2)\}=\exp\{-n\Psi_1(x)\}$ with $\Psi_1(x)=\Psi(x)-\frac 1{2m}\norm{x-\mu^*}_{\Sigma^*}^2$. We then have
\begin{align}
D^{a(\ell_u)}\ggis(x)&=(2\pi)^{s/2}|\Sigma|^{1/2}\sum_{b+c=a(\ell_u)}D^b \tilde h_n(x)D^c g_2(x)\\
&=(2\pi)^{s/2}|\Sigma|^{1/2}\tilde h_n(x)\sum_{b+c=a(\ell_u) \atop P\in \Pi(b)}(-n)^{|P|}D^c g_2(x)\prod_{ \beta\in  P}D^{\beta}\Psi_1(x),
\end{align}
and similar to \eqref{eq:fun1}
\begin{align*}
    F_u(n)&\lesssim \norm{L}_{\max}^{2|u|}|\Sigma|^{1/2} \sum_{\ell_u\in \{1{:}s\}^{|u|}} \sum_{b+c=a(\ell_u) \atop P\in \Pi(b)}n^{2|P|}\tilde J_{n}\\
    &\lesssim n^{-|u|-s/2}\sum_{\ell_u\in \{1{:}s\}^{|u|}}\sum_{b+c=a(\ell_u) \atop P\in \Pi(b)}n^{2|P|}\tilde J_{n},
\end{align*}
where 
\begin{equation}
    \tilde J_{n} = \int_{\R^s}\exp\{-2n\Psi_2(x)\}\left(D^c g_2(x)\prod_{ \beta\in  P}D^{\beta}\Psi_1(x)\right)^2dx,
\end{equation}
and $$\Psi_2(x)=\Psi_1(x)-\norm{x-\mu^*}_{\Sigma^*}^2/(2\alpha)=\Psi(x)-\frac12(1/m-1/\alpha)\norm{x-\mu^*}_{\Sigma^*}^2.$$
Since $\mu^*$ attains the minima of $\Psi_2(x)$, $\Psi_2(\mu^*)=\Psi(\mu^*)=0$.
Let $$Q_2(x) =\left(D^c g_1(x)\prod_{ \beta\in  P}D^{\beta}\Psi_1(x)\right)^2.$$
Then similarly,
\begin{align*}
n^{2|P|}\tilde J_{n} &=n^{2|P|}\int_{\R^s}\exp\{-2n\Psi_2(x)\}Q_2(x)dx\lesssim n^{|b|-s/2},
\end{align*}
where we use the fact that $Q_2(x)$ possesses a zero of order $2|P|_1$.
This leads to
\begin{align*}
    F_v(n)\lesssim n^{-|v|-s/2}n^{|b|-s/2}
    \le n^{-s/2-s/2}.
\end{align*}
We thus arrive at $\|\gis(z)\|_{\mathcal{F}}^2=O(n^{-s})$. 
\end{proof}

\begin{corollary}\label{cor:rates}
Suppose that $f(z)$ and $\Psi(z)$ satisfy  Assumption~\ref{assum3s} with  $M=M_f$ and  $M=M_\Psi$, respectively, and Assumption~\ref{5} are satisfied for $\Psi(z)$. If $\mu$ and $\Sigma$ do not depend on $n$ and 
$$\gamma_{n,1}:=\frac{n\delta\lambda_{\min}(\Sigma)}{\lambda_{\max}(\Sigma^*)}+\frac{\lambda_{\min}(\Sigma)}{\lambda_{\max}(\Sigma_0)}+2(M_f+sM_\Psi)\lambda_{\min}(\Sigma)>\frac{1}{2},$$
then for any arbitrarily small $\epsilon>0$,
$$\sqrt{E^\Delta |I_{s,\phi}(\gis)-\hat I_{s}(\gis\circ \Phi^{-1})|^2}\le C_{\epsilon,s}n^{s/4}N^{-\min(\gamma_{n,1},1)+\epsilon},$$
where $C_{\epsilon_n,s}>0$ depends only on $\epsilon$ and $s$. Assuming that $\mu=\mu^*$ and $\Sigma=(m/n)\Sigma^*$ and $F(z)=2\Psi(z)-(1/m-1/\alpha)\norm{z-\mu^*}_{\Sigma^*}^2$ has a global minimizer $\mu^*$, and 
$$\gamma_{n,2}:=1+\delta-\frac{1}{m}+\frac{\lambda_{\max}(\Sigma^*)}{n\lambda_{\max}(\Sigma_0)}-\frac{2\lambda_{\max}(\Sigma^*)(M_f+sM_\Psi)}{n}>\frac{1}{2},$$
then $$\sqrt{E^\Delta |I_{s,\phi}(\gis)-\hat I_{s}(\gis\circ \Phi^{-1})|^2}\le C_{\epsilon,s}n^{-s/2}N^{-\min(\gamma_{n,2},1)+\epsilon}.$$
\end{corollary}
\begin{proof}

We keep using the notations in Theorem~\ref{thm9}. First, we consider the case of $\mu$ and $\Sigma$ independently of $n$. Note that $D^c g_1(z)$
    is a sum of finite terms of the form
\begin{align*}
    \prod_{i=1}^s z_i^{t_i}W(z)D^a f(z).
\end{align*}
As a result, $Q_1(z)$ is a sum of finite terms of the form
$$T_1(z) = \exp\{-\norm{z-\mu}_{\Sigma}^2/\alpha\}W(z)^2\left(\prod_{i=1}^s z_i^{2t_i}\right)D^a f(z)D^b f(z)\left(\prod_{ \beta\in  P}D^{\beta}\Psi(z)\right)^2.$$
Since $f(z)$ and $\Psi(z)$ satisfy  Assumption~\ref{assum3s} with  $M=M_f$ and  $M=M_\Psi$, we have
$$|T_1(z)|\lesssim  \prod_{i=1}^s |z_i|^{2t_i}\exp\{(1-1/\alpha)\norm{z-\mu}_{\Sigma}^2 - \norm{z-\mu_0}_{\Sigma_0}^2+2(M_f+sM_\Psi)\norm{z}^2\}.$$
To satisfy Assumption \ref{Convex}  with $F(z)=2\Psi(z)$ and $Q(z)=Q_1(z)$, it suffices to verify
$$\int_{\R^s}\exp\{-2n\Psi(z)\}|T_1(z)|dz<\infty,\ \forall n>0.$$
Under Assumption~\ref{5}, we have
\begin{align*}
\int_{\R^s}\exp\{-2n\Psi(z)\}|T(z)|dz
&\le \int_{\R^s}\prod_{i=1}^s |z_i|^{2t_i}\exp\{-n\delta\norm{z-\mu^*}_{\Sigma^*}^2+(1-1/\alpha)\norm{z-\mu}_{\Sigma}^2\\
&\quad\quad\quad - \norm{z-\mu_0}_{\Sigma_0}^2+2(M_f+sM_\Psi)\norm{z}^2\}dz,
\end{align*}
which is bounded if 
$$-n\delta/\lambda_{\max}(\Sigma^*)+(1-1/\alpha)/\lambda_{\min}(\Sigma)-1/\lambda_{\max}(\Sigma_0)+2(M_f+sM_\Psi)<0,$$
or equivalently,
$$1-1/\alpha<\frac{n\delta\lambda_{\min}(\Sigma)}{\lambda_{\max}(\Sigma^*)}+\frac{\lambda_{\min}(\Sigma)}{\lambda_{\max}(\Sigma_0)}-2(M_f+sM_\Psi)\lambda_{\min}(\Sigma):=\gamma_{n,1}.$$
Applying Theorem \ref{thm2.2} with $\lambda=r-\epsilon/2=1-1/\alpha-\epsilon/2$ and $1-1/\alpha=\min(1,\gamma_{n,1})-\epsilon/2$ completes the first part of the proof.

Now we assume $\mu=\mu^*$ and $\Sigma=(m/n)\Sigma^*$. Similarly, due to $\pi_0(z)\propto \exp\{-(1/2)\norm{z-\mu_0}^2_{\Sigma_0}\}$, $D^c g_2(z)$
    is a sum of finite terms of the form
\begin{align*}
    \prod_{i=1}^s z_i^{t_i}\pi_0(z)D^a f(z).
\end{align*}
As a result, $Q_2(z)$ is a sum of finite terms of the form
$$T_2(z) = \left(\prod_{i=1}^s z_i^{2t_i}\right)\exp\{-\norm{z-\mu_0}^2_{\Sigma_0}\}D^a f(z)D^b f(z)\left(\prod_{ \beta\in  P}D^{\beta}\Psi_1(z)\right)^2.$$
Note that $$D^a\Psi_1(z)=D^a\Psi(z)+(1/2)D^a(\norm{z-\mu^*}_{\Sigma^*}^2).$$
Since $D^a\Psi(x)=O(\exp\{M_\Psi\norm{z}^2\})$, $D^a\Psi_1(z)=O(\exp\{M_\Psi\norm{z}^2\})$, implying   Assumption~\ref{assum3s} holds also for $\Psi_1(x)$ with  $M=M_\Psi$.
We thus have
$$|T_2(z)|\lesssim  \prod_{i=1}^s |z_i|^{2t_i}\exp\{-\norm{z-\mu_0}_{\Sigma_0}^2 +2(M_f+sM_\Psi)\norm{z}^2\}.$$
To satisfy Assumption \ref{Convex}  with $F(z)=2\Psi(z)-(1/m-1/\alpha)\norm{z-\mu^*}_{\Sigma^*}^2$ and $Q(z)=Q_2(z)$, it suffices to verify
$$\int_{\R^s}\exp\{-2nF(z)\}|T_2(z)|dz<\infty,\ \forall n>0.$$
Under Assumption~\ref{5}, we have
\begin{align*}
\int_{\R^s}\exp\{-2nF(z)\}|T_2(z)|dz
&\le \int_{\R^s}\prod_{i=1}^s |z_i|^{2t_i}\exp\{-n(\delta-1/m+1/\alpha)\norm{z-\mu^*}_{\Sigma^*}^2\\
&\quad\quad\quad - \norm{z-\mu_0}_{\Sigma_0}^2+2(M_f+sM_\Psi)\norm{z}^2\}dz.
\end{align*}
It thus suffices to ensure that
$$-n(\delta-1/m+1/\alpha)/\lambda_{\max}(\Sigma^*)-1/\lambda_{\max}(\Sigma_0)+2(M_f+sM_\Psi)<0,$$
or equivalently,
$$1-1/\alpha<1+\delta-1/m+\frac{\lambda_{\max}(\Sigma^*)}{n\lambda_{\max}(\Sigma_0)}-\frac{2\lambda_{\max}(\Sigma^*)(M_f+sM_\Psi)}{n}=:\gamma_{n,2}.$$
Applying Theorem \ref{thm2.2} with $\lambda=r-\epsilon/2=1-1/\alpha-\epsilon/2$ and $1-1/\alpha=\min(1,\gamma_{n,2})-\epsilon/2$ completes the second part of proof.
\end{proof}

\begin{remark}\label{rem:rate}
   As $\gamma_{n,1} \to \infty$ when $n \to \infty$ and $\delta > 0$, both PriorIS and ODIS methods have RMSEs of $O(n^{s/4}N^{-1+\epsilon})$. In contrast, we have $\gamma_{n,2} \to 1 + \delta - 1/m$ as $n \to \infty$. To ensure $1 + \delta - 1/m > 1/2$, it is necessary that $m > 1/(1/2 + \delta)$. 
   In the case of LapIS, where $m=1$, the RMSE approaches $O(n^{-s/2}N^{-\delta+\epsilon})$ for large $n$ and $\delta > 1/2$. If the exact value of $\delta$ is unknown, a conservative approach is to choose $m>2$ (e.g., $\Sigma = \frac{4}{n}\Sigma^*$), excluding LapIS. For known $\delta > 0$, we suggest setting $\mu=\mu^*$ and $\Sigma = \frac{1}{\delta n}\Sigma^*$ (i.e., $m = 1/\delta$), yielding an RMSE close to $O(n^{-s/2}N^{-1+\epsilon})$. We refer to this adjusted method of IS as NewIS in subsequent discussions.
\end{remark}

\begin{remark}
   In accordance with Theorem~\ref{thm1}, the estimand $I_{s,\phi}(\gis)$ is of the order $O(n^{-s/2})$ as $n \rightarrow \infty$. Consequently, it is more appropriate to consider the relative RMSE (RRMSE) defined as
   \begin{equation*}
   \mathrm{RRMSE} := \frac{\sqrt{E^\Delta |I_{s,\phi}(\gis)-\hat I_{s}(\gis \circ \Phi^{-1})|^2}}{I_{s,\phi}(\gis)}
   \end{equation*}
rather than RMSE when investigating the asymptotic behavior with respect to $n$.
   By Corollary~\ref{cor:rates}, the RRMSE 
   for ODIS and PriorIS is  $O(n^{3s/4})$, while the  RRMSE for LapIS and NewIS is  $O(1)$ as $n\to \infty$. From this point of view, the later case is robust with respect to the noise level $n$. The results can also be applied for the ratio estimator \eqref{31} by using \eqref{eq:boundedratio}. 
\end{remark}


\section{Results for $t$-distribution proposals}\label{sec:tdist}
It is well known that the tail of Gaussian distribution are short and sometime it will make the likelihood function more severe. 
In this section, we employ a proposal that results from a linear transformation of i.i.d. $t$-distributions with the degree of freedom $\nu>0$. It has the effect of changing the density $\phi(t)$ from standard Gaussian distribution $N(0,1)$ in Section~\ref{sec:gaussian} to $t$-distribution $t_\nu$. We now  work on the RKHS $\mathcal{F}$ with the pairing  
\begin{align*}
    \phi(t) &=\frac{\Gamma(\frac{\nu+1}{2})}{\sqrt{\nu\pi}\Gamma(\frac{\nu}{2})}\left(1+\frac{t^2}{\nu}\right)^{-\frac{\nu+1}{2}},\\
    \psi(t) &=(1+|t|)^{-\alpha},
\end{align*}
where $2\alpha+1<\nu$. For a positive definite matrix $\Sigma$, the proposal for IS is
$$q(x) = |\Sigma|^{-1/2}\phi(L^{-1}(x-\mu)),$$ where $LL^T=\Sigma$. The numerator in \eqref{3} is rewritten as $\mathbb{E}[\gtis(Z)]$
where $Z_i\simiid t_v$,
$$\gtis(z) =G_0 (x)W_t(x)\text{ with }x=\mu+Lz,$$
and the likelihood function
\begin{equation}
\begin{aligned}\label{ratio}
    W_t(x) &=\frac{\pi_0(x)}{q(x)}=\frac{|\Sigma|^{1/2}}{(2\pi)^{s/2}|\Sigma_0|^{1/2}}\exp\left\{-\frac{1}{2}\norm{\mu +Lz-\mu_0}_{\Sigma_0}^2\right\}\prod_{i=1}^{s}\frac{\sqrt{\nu\pi}\Gamma(\frac{\nu}{2})}{\Gamma(\frac{\nu+1}{2})}\left(1+\frac{z_i^2}{\nu}\right)^{\frac{\nu+1}{2}}.
\end{aligned} 
\end{equation}

\begin{theorem}\label{thm:5.3}
    Let $\gtis(z)=G_0(\mu+Lz)W_t(\mu+Lz)$ with $W_t(x)$ given by \eqref{ratio} and $\Sigma=LL^T$. Suppose that $G_0(z)$ satisfies Assumption~\ref{assum3s} with the growth rate $M\in\R$. If 
\begin{equation}\label{eq:Mcond}
   \frac{1}{\lambda_{\max}(\Sigma_0)}-2M > 0, 
\end{equation}
then a randomly shifted lattice rule using in the estimator $\hat{I}_{s}(\gtis\circ\Phi^{-1})$ with $N$ points can be constructed by a CBC algorithm with an RMSE of $O(N^{-1+ \frac{1}{2\nu}+ \epsilon})$ for arbitrary small $\epsilon > 0$,  where $\Phi(t)$ denotes the CDF of the $t$-distribution $t_v$.
\end{theorem}
\begin{proof}
Recall that $x=\mu+Lz$. By the Faa di Bruno formula, and the form of $W_t(x)$ given by \eqref{ratio}, $\partial^u \gtis(z)$ is a sum of finite terms of the form
\begin{align*}
    \tilde{T}(z) = \prod_{i=1}^s z_i^{t_i}\left(1+\frac{z_i^2}{\nu}\right)^{t_i^{'}}D^a G_0(x)\exp\left(-\frac{1}{2}\norm{x-\mu_0}_{\Sigma_0}^2\right)
\end{align*}
for $a\in\mbn_0^s$ satisfying $|a|\le s$ and some integers $0\le t_i\le s$, $t_i^{'}\in\{\frac{\nu-1}{2},\frac{\nu+1}{2}\}$.
By Assumption~\ref{assum3} and using \eqref{W(z)}, we have
\begin{align*}
    |\tilde{T}(z)| &\lesssim\left(\prod_{i=1}^s z_i^{t_i}\left(1+\frac{z_i^2}{\nu}\right)^{t_i^{'}}\right)\exp\left\{M\|x\|^2 - \frac{1}{2}\norm{x-\mu_0}_{\Sigma_0}^2\right\}.
\end{align*}
If $2MI_s - \Sigma_0^{-1}$ is negative definite, then  $\tilde{T}(z)$ is bounded, implying
\begin{equation*}
    S := \int_{\R^u}\left(\int_{\R^{-v}}\tilde T(z)\phi_{-u}(z)dz_{-u}\right)^2\psi_u(z)^2dz_u<\infty
\end{equation*}
for any $\alpha>0$ satisfying $2\alpha +1<\nu$. 
To this end, it suffices to take
\begin{align*}
    2M -\frac 1{\lambda_{\max}(\Sigma_0)}< 0.
\end{align*}
Applying Theorem \ref{thm2.2} with $\lambda=r-\epsilon/2=1-\frac{2\alpha+1}{2\nu}-\epsilon/2$ and $\alpha=\nu\epsilon/2$ completes the proof.
\end{proof}

Note that the condition  \eqref{eq:gammacond} in Theorem \ref{ISqmc} actually implies the condition \eqref{eq:Mcond}, which does not depend on the choice of $\mu$ and $\Sigma$. 
The RMSE rate established in Theorem~\ref{thm:5.3} holds for the choices of $\mu$ and $\Sigma$ from PriorIS, ODIS, and LapIS.
Moreover, the condition \eqref{eq:Mcond} is easily satisfied for a large enough noise level $n$ as suggested by \eqref{thm:g0}.
In the LapIS case of Gaussian proposal, however, the RMSE rate  is relevant to the covariance matrix $\Sigma=n^{-1}\Sigma^*$and deteriorates as $n$ goes up. From this point of view,  one may prefer to use $t$-distributions rather than Gaussian distributions as the proposal. According to Theorem~\ref{thm9}, the concentration effect of the norm $\norm{\gtis}_{\mathcal{F}}$ can be studied similarly.

\section{Numerical Experiments}\label{sec4}
In this section, we present two examples illustrating our previous theoretical results for the combination of IS and QMC. In the context of QMC estimators, we employ a lattice rule featuring product and order dependent weight parameters accompanied by a a single random shift with uniform distribution. The generating vector implemented in this rule is procured from Dirk Nuyens's  website \url{https://bitbucket.org/dnuyens/qmc-generators/src/master/LATSEQ/exod2_base2_m20.txt}, designated as an embedded lattice rule \cite{cools2006constructing}. We point out here that this generating vector is a standard, off-the-shelf generating vector, rather than a generating vector specifically constructed for the weights implicitly defined in in Theorem \ref{thm2.2}. In practice we found this generating vector to work well, even though the convergence rates Theorem \ref{thm2.2} were not proved for this particular choice.

\subsection{A toy example}
Consider the model $\mathcal{G}(z) = Az+\tau F(z)$ as discussed in Example \ref{exam:nonlinear} with $\tau >0$, $A = I_s$ and $F(z) =(z_1e^{-z_1^2},\dots,z_se^{-z_s^2})^T$. For simplicity, we assume that $\Gamma=I_s$ and $y=0$, leading to  $$\Psi(z)=\frac{1}{2}\norm{\mathcal{G}(z)}^2=\frac{1}{2}\sum_{i=1}^sz_i^2(1+\tau e^{-z_i^2})^2.$$
It is easy to see that $\mu^*=\bm 0$ is the global minimizer of $\Psi(z)$. Since $\nabla^2\Psi(\mu^*)=(1+\tau)^2I_s$, $\Sigma^*=(\nabla^2\Psi(\mu^*))^{-1}=(1+\tau)^{-2}I_s$. Note that $\Psi(z)\ge \frac{1}{2}\norm{z}^2$. To satisfy Assumption~\ref{5}, i.e.,
$$\Psi(z)\ge\frac{\delta}2\norm{z}^2_{\Sigma^*}=\frac{\delta(1+\tau)^2}2\norm{z}^2,$$
it suffices to take $\delta= (1+\tau)^{-2}$. In our experiments, we take $\delta\in\{1/4,1/2,3/4,1\}$ and the dimension $s = 8$. For the Gaussian prior $N(\mu_0,\Sigma_0)$, we take $\mu_0=\bm 1$ and $\Sigma_0$ with entries $\max(i,j)$ for $i,j=1,\dots,s$. 
We consider four Gaussian proposals for IS:
\begin{itemize}
    \item PriorIS: $\mu = \mu_0=\bm 1$ and $\Sigma=\Sigma_0$,
    \item ODIS: $\mu = \mu^*= \bm 0$ and $\Sigma=\Sigma_0$,
    \item LapIS: $\mu =\mu^*= \bm 0$ and $\Sigma=\frac 1n\Sigma^*=\frac{\delta}{n} I_s$,
    \item NewIS: $\mu =\mu^*= \bm 0$ and $\Sigma=\frac{1}{\delta n
    }\Sigma^*=\frac 1 nI_s$.
\end{itemize}
The test function we use is $f(z)= \|z\|$. It is easy to see that Corollary \ref{cor:rates} holds with arbitrarily small $M_f,M_\Psi>0$. 

When $s=2$ and $\delta=3/4$, Figure \ref{figure22} illustrates the concentration effect of the posterior for $n\in\{1, 10^2,10^4\}$, and the resulting transformed posterior $\pi_0(\mu+Lx)\exp({-n\Psi(\mu+Lx)})$ using NewIS is also provided. 
We only provide numerical results for estimating the  numerator in \eqref{3}. Since the numerator and  denominator in \eqref{3} are of the order $O(n^{-s/2})$, the RMSEs of IS estimators for the numerator are scaled by a factor of $n^{s/2}$ for investigating the robustness of the noise level.
The RMSEs are estimated based on $40$ independent replications.

Initially, we use the four Gaussian proposals for MC and QMC methods with different values of $\delta$ to investigate the relationship between RMSE and the noise level $n$,  where the sample size is fixed at $N = 2^{14}$. The results are presented in Figure \ref{figure1}.  We observe that both LapIS and NewIS  in the MC or QMC setting enjoy the noise-level-robustness for different cases of $\delta$. However,  PriorIS and ODIS deteriorate as the noise level $n$ increases. These agree with the theoretical results in Corollary~\ref{cor:rates}.

\begin{figure}[ht]
\centering  

{%
\label{Fig.sub.21}
\includegraphics[width=0.30\textwidth]{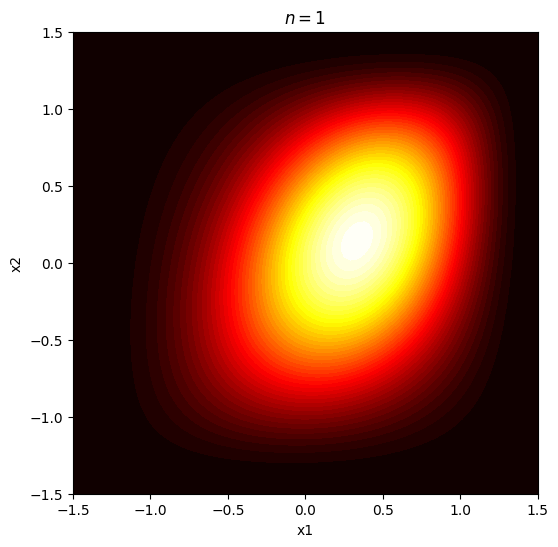}} 
{%
\label{Fig.sub.22}
\includegraphics[width=0.30\textwidth]{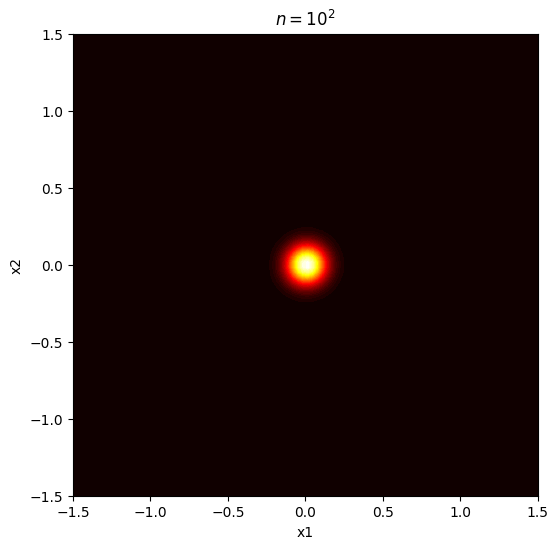}}
{%
\label{Fig.sub.23}
\includegraphics[width=0.30\textwidth]{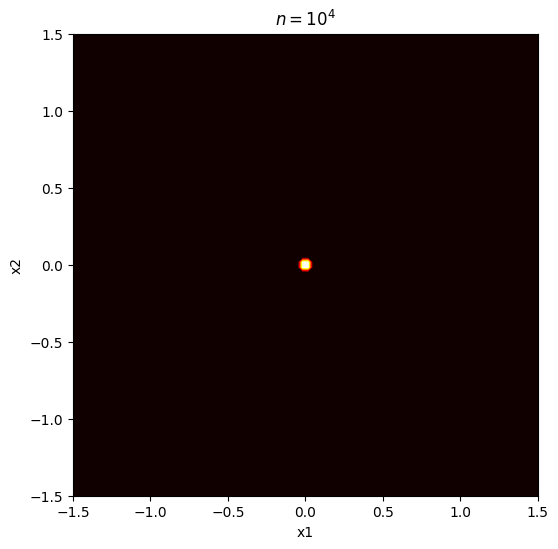}}

\vspace{0.3cm} 

{%
\label{Fig.sub.21-2}
\includegraphics[width=0.30\textwidth]{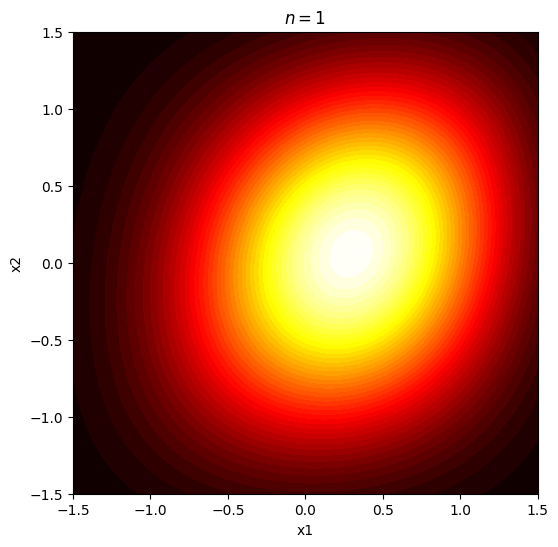}}
{%
\label{Fig.sub.22-2}
\includegraphics[width=0.30\textwidth]{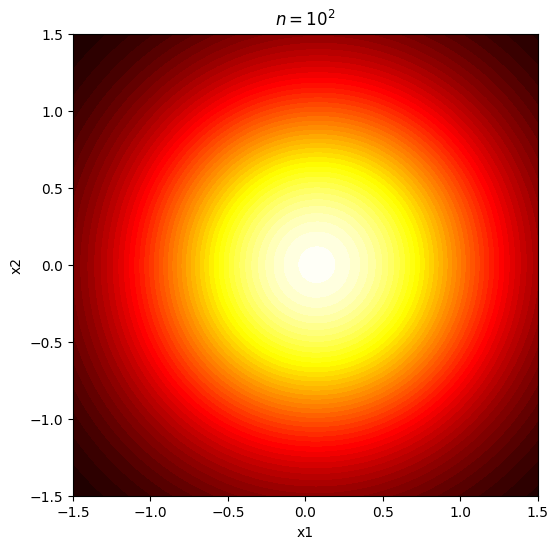}}
{%
\label{Fig.sub.23-2}
\includegraphics[width=0.30\textwidth]{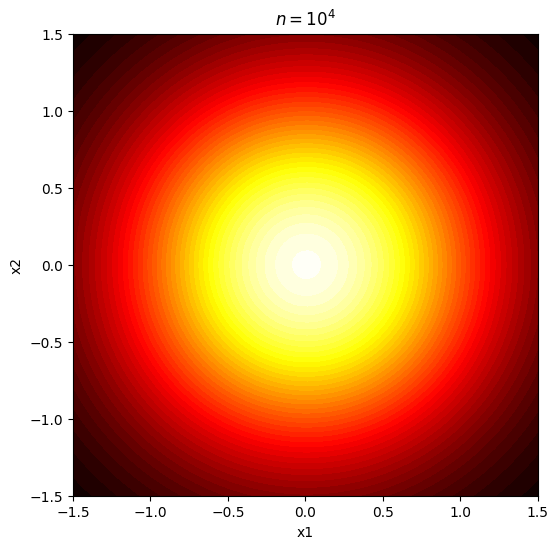}}
\caption{The first row shows the posterior distribution of the toy example for $n=1,10^2,10^4$, respectively. The second row
shows the  transformed posteriors $\pi_0(\mu+Lx)\exp({-n\Psi(\mu+Lx)})$ using NewIS. The dimension $s = 2$ and $\delta=3/4$.}\label{figure22}
\end{figure}

Subsequently, with the noise level fixed at $n=2000$, we explore the relationship between RMSE and the sample size $N$, as shown in Figure \ref{figure2}. 
We find that LapIS and NewIS contribute to a more significant enhancement in the setting of MC and QMC compared to PriorIS and ODIS. Particularly, the convergence speed of NewIS in QMC is faster than that of LapIS for $\delta<1$.  Moreover, the convergence rate of the QMC-based LapIS method improves as $\delta$ increases. This observation aligns with our theoretical results (see Remark~\ref{rem:rate}), suggesting that if we know the value of $\delta$, we can choose a more suitable Gaussian proposal  to achieve a higher order of convergence.

Finally, Figure~\ref{small_sacle} presents RMSEs for a small noise level $n=20$ and $\delta=1/4$. For this case, the posterior is far away from Gaussian distribution. We observe that LapIS performs much worse than ODIS and NewIS in QMC.
This suggests that LapIS may result in a severe unbounded integrand for QMC as discussed in Remark~\ref{rem:qmcrate}.

\begin{figure}[ht]
\centering  

{%
\label{Fig.sub.8}
\includegraphics[width=0.40\textwidth]{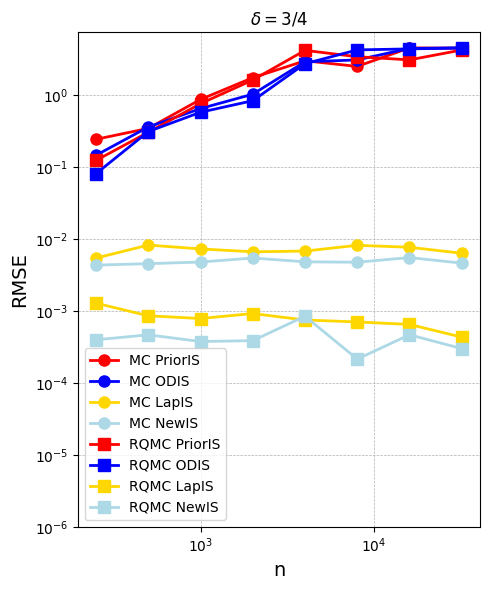}} 
{%
\label{Fig.sub.9}
\includegraphics[width=0.40\textwidth]{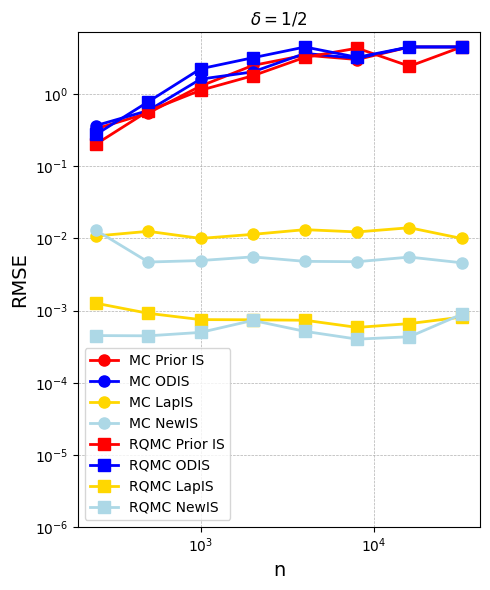}}

\vspace{0.3cm} 

{%
\label{Fig.sub.10}
\includegraphics[width=0.40\textwidth]{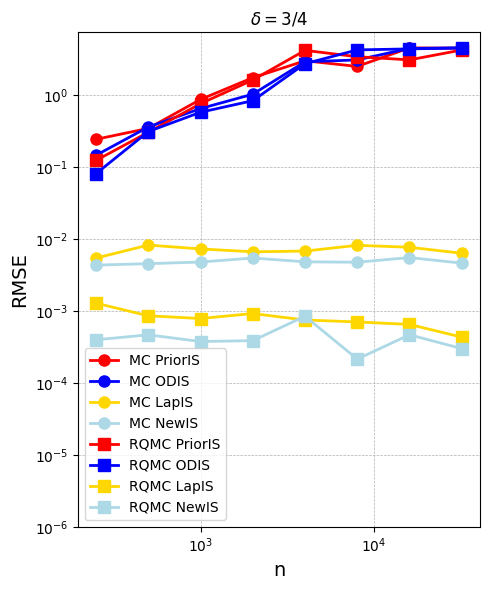}}
{%
\label{Fig.sub.11}
\includegraphics[width=0.40\textwidth]{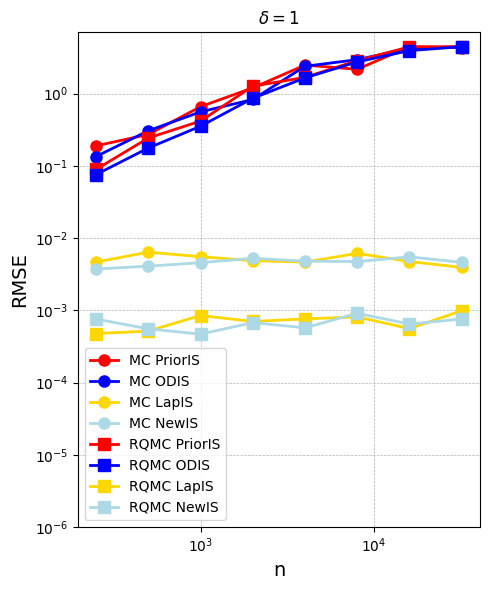}}

\caption{Scaled RMSE by a factor of $n^{s/2}$ for MC and QMC methods 
 with PriorIS, ODIS, LapIS, and NewIS for the toy model. The sample size $N=2^{14}$.}\label{figure1}
\end{figure}





\begin{figure}[ht]
\centering  

{%
\label{Fig.sub.8-2}
\includegraphics[width=0.40\textwidth]{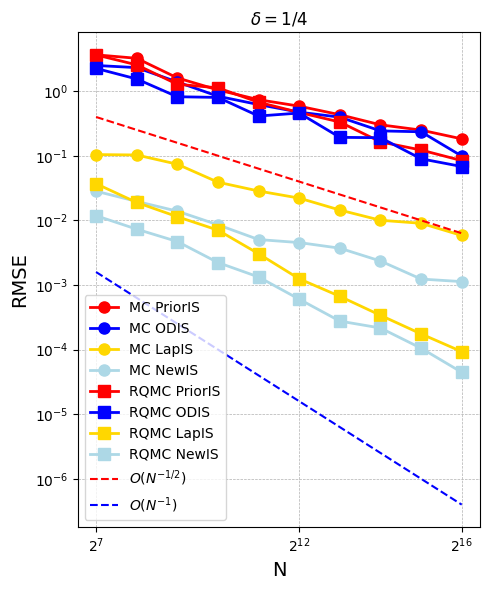}} 
{%
\label{Fig.sub.9-2}
\includegraphics[width=0.40\textwidth]{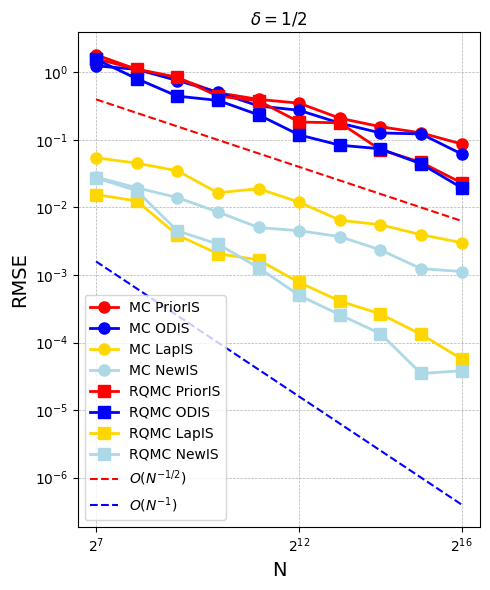}}

\vspace{0.3cm} 

{%
\label{Fig.sub.10-2}
\includegraphics[width=0.40\textwidth]{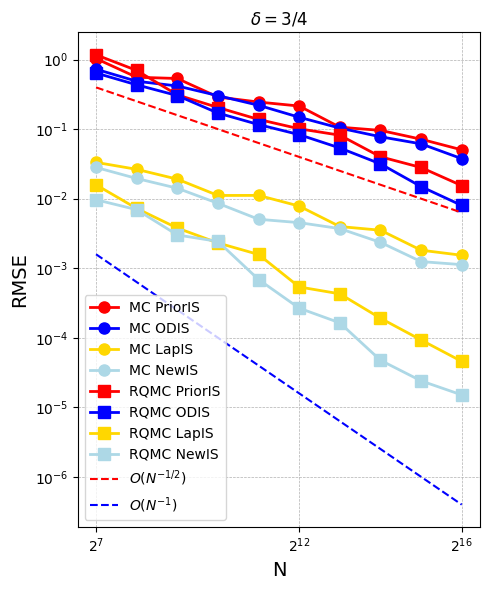}}
{%
\label{Fig.sub.11-2}
\includegraphics[width=0.40\textwidth]{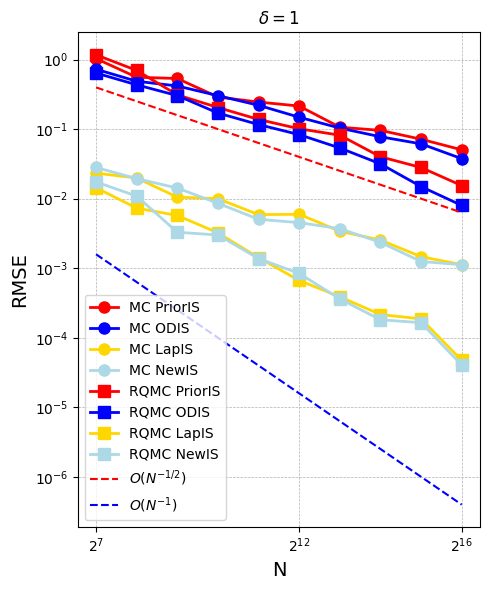}}

\caption{Scaled RMSE by a factor of $n^{s/2}$ for MC and QMC methods with PriorIS, ODIS, LapIS, and NewIS for the toy model. The noise level $n=2000$.}\label{figure2}
\end{figure}
\begin{figure}[ht]
\centering  
\includegraphics[width=0.4\textwidth]{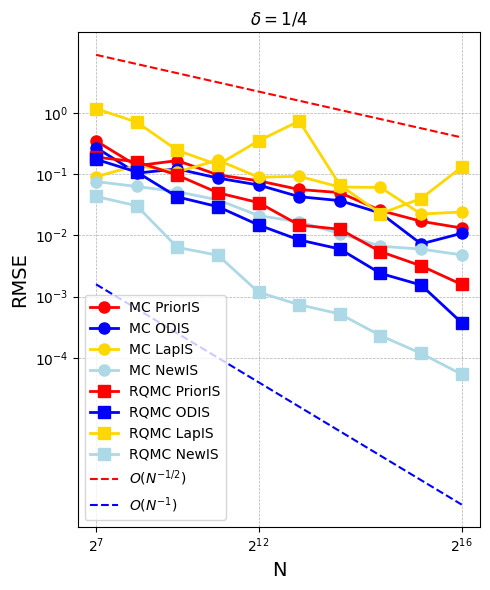} 
\caption{Scaled RMSE by a factor of $n^{s/2}$ for MC and QMC methods  with PriorIS, ODIS, LapIS, and NewIS. The noise level $n=20$ and $\delta=1/4$.}\label{small_sacle}
\end{figure}

\subsection{Parameterized PDEs}
Consider the model \eqref{eq:EPDE} with $g(x) = 100x$, $D=[0,1]$, and the diffusion coefficient
\begin{equation*}
   a(x,w) = a_s(x,z)=\exp\biggl(\sum_{j=1}^{s}z_{j}\xi_j(x)\biggr), 
\end{equation*}
where $\xi_j(x)=(0.1/j)\sin(j\pi x)$ which comes from the PDE model in \cite{bib1},
and  $z_k \in \R$ for $k = 1,...,s$, are to be inferred based on noisy observations of the solution $q$ at certain points in $[0,1]$. 
These observations are taken at $t \in \{0.125, 0.25, 0.375, 0.5, 0.625, 0.75, 0.875\}$, implying $\mathcal{G}: \mathbb{R}^s \rightarrow \mathbb{R}^7$. We assume additive Gaussian observational noise, where $\eta$ follows a normal distribution $N(0,\Gamma_n)$, with noise covariance $\Gamma_n = n^{-1}\Gamma$, where $\Gamma = I_7$.
In the following analysis, we impose a Gaussian prior $\pi_0=N(0,I_s)$ and aim to integrate with respect to the resulting posterior measure. Here we take a simple test function $f(z) = \norm{z}$ and $s = 8$ for illustration. The observational data $y$ is a solution of the PDE with the fixed the parameter, a vector of all ones. The way to solve the PDE is the finite difference method with the step equals to $1/64$. For details about the finite difference method, we refer to \cite{grossmann2007numerical}.

\begin{figure}[ht]
\centering
{
\label{Fig.sub.1}
\includegraphics[width=0.4\textwidth]{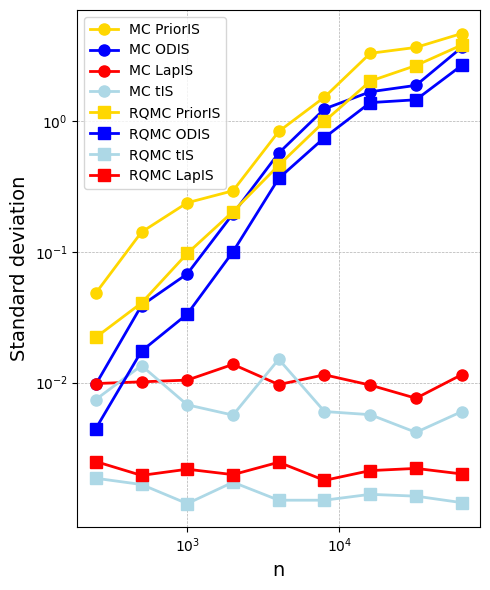}}
{
\label{Fig.sub.2}
\includegraphics[width=0.4\textwidth]{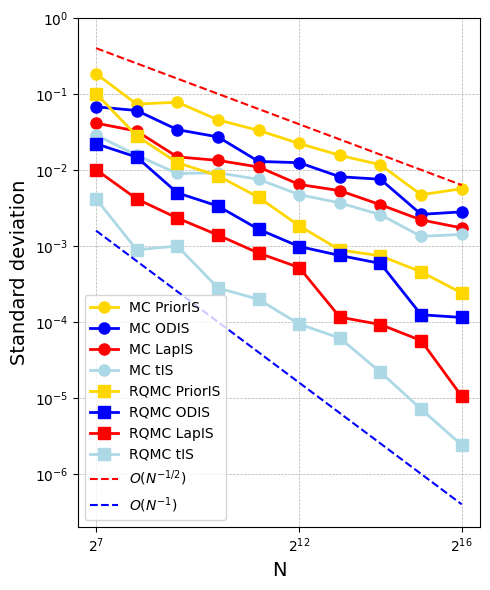}}
\caption{The left panel reports the standard deviation of the ratio estimator \eqref{31} with respect to $n$ for MC and QMC methods with PriorIS, ODIS,
LapIS, and tIS when the sample size is  fixed at $N=2^{14}$. The right panel shows the standard deviation of the ratio estimator \eqref{31} with respect to $N$ when the noise level is fixed at $n=2000$. The degrees of freedom for tIS is $\nu=5$.} \label{figure3}
\end{figure}

We apply MC and QMC with different IS methods (PriorIS, ODIS, LapIS, and tIS) to estimate the posterior expectation \eqref{2}. The tIS method is a variant of LapIS by using the $t$ distribution with the degrees of freedom $\nu=5$ instead of the standard Gaussian distribution in the sampling. The standard deviations of the ratio estimator \eqref{31} are presented in Figure~\ref{figure3} with varying $n$ or $N$. The sample size $N=2^{14}$ and $n=2000$ are fixed for the left and right panels of Figure \ref{figure3}, respectively. The results are based on 40 independent replications.   It is evident that while both tIS and LapIS exhibit robustness with respect to the noise level variability, the variance reduction achieved with the $t$-distribution as the proposal is more substantial compared to that with the Gaussian distribution. Consequently, the use of $t$-distribution in the context of inverse problems rooted in PDE scenarios continues to hold significant potential.

\section{Conclusions}\label{sec:conclusion}
    The theoretical analysis is provided for the application of QMC combined with IS in BIPs. Firstly, we analyzed some convergence results of IS combined with the lattice rule and observed that choosing an appropriate importance density can ensure that the function lies in the space of weighted functions. On the other hand, selecting a poor importance density can make the function unfriendly to QMC methods. Then we proposed an efficient IS method based on Gaussian proposal that not only enjoys the noise-level-robustness property but also has a faster convergence rate. The proposed IS depends on a prior information of the lower bound of the negative log-likelihood (i.e., Assumption~\ref{eq:remainer}). On the other hand, we found that using $t$-distribution proposals in LapIS performs better than using Gaussian proposals.
    
    It should be noted that this paper does not delve deeply into the relationships among convergence rate, dimensionality, and degrees of freedom of the $t$-distribution. The implied constants in the error bounds of this paper may still be associated with the dimensionality or the degrees of freedom of the $t$-distribution. To obtain dimension-independent results as the dimension tends to infinity, one may select some special weights involved in the RKHS. A related approach is proposed in \cite{kuo2012quasi}. The study of noise-level robustness in infinite dimensions is a direction for future research.
\section{Acknowledgments}
The authors would like to thank Professor Nicolas Chopin, Jiarui Du and Xinting Du for helpful comments and suggestions.

\bibliographystyle{siamplain}

\bibliography{references}

\end{document}